\documentclass[11pt,a4paper]{article}
\usepackage[left=2cm, right=2cm, top=2cm]{geometry}

\usepackage{float}
\usepackage{graphicx}
\usepackage{amsmath}
\usepackage{mathtools}
\usepackage{amssymb}
\usepackage{bm}
\usepackage{framed}
\usepackage{subfig}
\usepackage{hyperref}
\usepackage{nomencl}
\usepackage{color}
\usepackage{textcomp}
\usepackage{graphics}
\usepackage{epsfig}
\usepackage{bm}
\usepackage{epstopdf}
\usepackage{amsthm}
\usepackage{float}
\usepackage{latexsym,amsfonts}
\usepackage{url}
\usepackage{longtable}
\usepackage[figuresright]{rotating}
\usepackage{listings}
\usepackage{algpseudocode}
\usepackage{footnote}
\usepackage[linesnumbered,ruled,vlined]{algorithm2e}

\usepackage{mathrsfs}

\usepackage{xcolor}

\usepackage{makeidx}
\usepackage{authblk}

\makeindex

\SetKwInput{KwInput}{Input}                % Set the Input
\SetKwInput{KwOutput}{Output}              % set the Output

\newtheorem{theorem}{Theorem}
\newtheorem{Remark}{Remark}
\newtheorem{Assumption}{Assumption}

\newcommand\myeq{\mathrel{\stackrel{\makebox[0pt]{\mbox{\normalfont\tiny def}}}{=}}}
\newcommand\myiid{\mathrel{\stackrel{\makebox[0pt]{\mbox{\normalfont\tiny i.i.d.}}}{\sim}}}

\definecolor{tocorrect}{rgb}{0.97, 0.04, 0.56}

\newcounter{lnote}

\usepackage{relsize}
\newcommand{\Epsilon}{{\mathlarger{\bm{\epsilon}}}}
\DeclareMathOperator*{\argmin}{arg\,min}

\DeclarePairedDelimiterX{\infdivx}[2]{(}{)}{%
  #1\;\delimsize\|\;#2%
}
\newcommand{\infdiv}{D_{\text{KL}}\infdivx}

\begin{document}

\title{Multilevel Double Loop Monte Carlo and Stochastic Collocation Methods with Importance Sampling for Bayesian Optimal Experimental Design}
\author[1]{Joakim Beck~\thanks{joakim.beck@kaust.edu.sa}}

\author[2]{Ben Mansour Dia~\thanks{mansourben2002@yahoo.fr}}

\author[1]{Luis F.R. Espath~\thanks{espath@gmail.com}}

\author[1,3]{Ra\'ul Tempone~\thanks{raul.tempone@kaust.edu.sa}}

\affil[1]{King Abdullah University of Science and Technology (KAUST), Computer, Electrical and Mathematical Science and Engineering Division (CEMSE), Thuwal, 23955-6900, Saudi Arabia}
\affil[2]{Center for Integrative Petroleum Research (CIPR), College of Petroleum Engineering and Geosciences, King Fahd  University of Petroleum and Minerals (KFUPM), Dhahran 31261, Saudi Arabia}
\affil[3]{Alexander von Humboldt Professor in Mathematics of Uncertainty Quantification, RWTH Aachen University, 52062 Aachen, Germany}

\maketitle

\begin{abstract}
An optimal experimental set-up maximizes the value of data for statistical inferences and predictions. The efficiency of strategies for finding optimal experimental set-ups is particularly important for experiments that are time-consuming or expensive to perform. For instance, in the situation when the experiments are modeled by Partial Differential Equations (PDEs), multilevel methods have been proven to dramatically reduce the computational complexity of their single-level counterparts when estimating expected values. For a setting where PDEs can model experiments, we propose two multilevel methods for estimating a popular design criterion known as the expected information gain in simulation-based Bayesian optimal experimental design. The expected information gain criterion is of a nested expectation form, and only a handful of multilevel methods have been proposed for problems of such form. We propose a Multilevel Double Loop Monte Carlo (MLDLMC), which is a multilevel strategy with Double Loop Monte Carlo (DLMC), and a Multilevel Double Loop Stochastic Collocation (MLDLSC), which performs a high-dimensional integration by deterministic quadrature on sparse grids. For both methods, the Laplace approximation is used for importance sampling that significantly reduces the computational work of estimating inner expectations. The optimal values of the method parameters are determined by minimizing the average computational work, subject to satisfying the desired error tolerance. The computational efficiencies of the methods are demonstrated by estimating the expected information gain for Bayesian inference of the fiber orientation in composite laminate materials from an electrical impedance tomography experiment. MLDLSC performs better than MLDLMC when the regularity of the quantity of interest, with respect to the additive noise and the unknown parameters, can be exploited.\\[5pt]


\end{abstract}

\section{Introduction}\label{sec:introduction}
Experiments are meant to provide meaningful information about selected quantities of interest. An experiment may assume different set-ups in a broad sense, and can be time consuming or expensive to perform. Therefore, the design of experiments plays an important role in improving the information gain of the experiment; a comprehensive review of utility functions and their computational algorithms for Bayesian optimal experimental design is available in \cite{ryan2016}. Bayesian optimal experimental design involves the task of designing experiments with the objective of maximizing the value of data for solving inverse problems, in a Bayesian sense. Recent work on Bayesian alphabetical optimal experimental design includes \cite{alex2016,AS2018,Attia2018,CASG2017,TLBH2013,WWJ2018}. A widely popular Bayesian information-theoretic utility function for nonlinear models known as Expected Information Gain (EIG) was introduced in 1956 by Lindley \cite{key13} for measuring the amount of information provided by an experiment. Lindley defined the EIG utility function as the average, with respect to our prior knowledge expressed through a prior probability density function (pdf), of the relative information entropy, based on the Shannon information entropy \cite{key12}, of the prior pdf with respect to the posterior pdf. The prior and posterior pdfs express our knowledge before and after performing the experiment, respectively. The relative information entropy is equivalent to the Kullback-Leibler divergence \cite{kullback1959,kullback1951} of the prior pdf from the posterior pdf. The EIG criterion is computationally challenging to compute since it is a \emph{nested} expectation of the form $\mathbb{E}[f_1(\bm{X}_1)/\mathbb{E}[f_2(\bm{X}_1,\bm{X}_2) \vert \bm{X}_1]]$, where $f_i$ are real-valued functions and $\bm{X}_i$ are random variables. As an alternative to direct estimation of EIG, a lower bound estimate can be used as a design criterion, as proposed by Tsilifis et al. \cite{Tsilifis2017}, it provides an approximate solution to the original design problem and requires less computational work to solve. In the present work, the goal is to develop computationally efficient estimators of the EIG criterion that should satisfy a specified accuracy requirement. The subsequent, albeit important, task of efficiently maximizing the expected information gain in order to find the \emph{most informative} experimental set-up is beyond the scope of our study. Efficient optimization strategies on continuous design spaces include stochastic gradient methods (e.g., \cite{Carlon2018,HM2014,huan}) and the approximate coordinate exchange algorithm (e.g., \cite{MN1995,OW2017}).

Ryan \cite{Ryan2003} applied a Double Loop Monte Carlo (DLMC) estimator for the EIG criterion, which entails applying Monte Carlo (MC) sampling for estimating the outer expectation and, for each outer sample, MC sampling is performed to estimate the inner expectation. The DLMC estimator is highly computationally taxing and yields a bias due to the inner averaging. The computational work required by DLMC can be reduced by approximating the inner expectation by the Laplace method \cite{LBP2009,key15,key52} instead of MC sampling, but comes at the price of a bias that is challenging to control. An alternative to the Laplace method is the Laplace-based importance sampling presented in \cite{RDP2015}, which dramatically improves the sampling efficiency for the inner MC sampling of DLMC. In \cite{BDELT2018}, an optimization strategy for estimating the EIG criterion with error control was derived for DLMC Laplace-based Importance Sampling (DLMCIS) that for a desired error tolerance minimizes the computational work, and it was shown that the estimator requires vastly fewer inner-loop samples compared to the standard DLMC, typically by orders of magnitude, for the same accuracy. Huan et al. \cite{huan} proposed an approximation of EIG by standard DLMC by first replacing the underlying model by a polynomial approximation, using polynomial chaos expansions with pseudo-spectral projection, of the model's input-output mapping. This approach requires substantially fewer model evaluations than DLMC, but the approximation error is not fully estimated nor controlled.

In this work, we improve upon DLMCIS \cite{BDELT2018} by employing multilevel techniques \cite{G2015}. Multilevel Monte Carlo (MLMC) estimators (e.g., \cite{CST2013,G2015,ali2016}) have been widely used for estimating expectations of quantities of interest that depends on the solution of Partial Differential Equations (PDEs) as it accelerates the computations of expectations by using control variates, which are based on successive differences of a sequence of increasingly-refined meshes, to reduce the variance of MC estimators. The EIG criterion is of a nested expectation form and only a few multilevel methods have been proposed for quantities of interest of such form; see, e.g., Section 9 of \cite{G2015} for nested MLMC simulation and \cite{giles2019multilevel} for nested MLMC for efficient risk estimation.
We propose two multilevel estimators for EIG: Multilevel Double Loop Monte Carlo (MLDLMC) and Multilevel Double Loop Stochastic Collocation (MLDLSC). MLDLMC uses a multilevel strategy with DLMCIS and the \emph{level} defines the resolution of the numerical discretization of the underlying model and the number of inner-loop samples. As for the work complexity of MLDLMC, we provide an upper bound of the total work for a prescribed desired accuracy. A recent work by Godas et al. \cite{goda2018} proposed, independently of this work, an MLMC estimator using the Laplace-based importance sampling for the EIG criterion. They propose an MLMC estimator with an antithetic technique and where the number of inner-loop samples defines the level. The problem setting between this work and that of \cite{goda2018} is significantly different as we consider the case where the solution to the underlying model needs to be approximated instead of being known exactly. Hence, our multilevel methods have a multilevel hierarchy that controls the number of inner samples as well as the mesh discretization.
MLDLSC, the other method proposed, uses Multi-Index Stochastic Collocation (MISC) \cite{ali2016b} for the outer expectation of EIG and full-tensor stochastic collocation, see, e.g., \cite{BNT2010,BTNT2012}, for approximating the inner expectations. MLDLSC provides an error control of the quantity of interest as well as reducing the computational work for a given accuracy requirement, compared to \cite{huan}, as it combines model evaluations at different mesh resolutions, with only a small number of those evaluations are on fine meshes, which are those that are more expensive to evaluate. MLDLSC uses the Laplace-based importance sampling \cite{BDELT2018} for the inner expectation.

To assess the computational efficiency of our proposed methods, we consider an electrical impedance tomography (EIT) problem in which we infer the angle of fibers in a composite laminate material. The composite laminate has four plies, and five electrodes are deployed on each side of the plate. Each ply of the composite laminate is an orthotropic layer with its fibers uniformly distributed along one predetermined direction. The electrodes inject electrical current and measure the electrical potential, which in turn is used to infer the material properties. We adopt the complete electrode model (CEM) \cite{somersalo} to simulate EIT experiments for composite laminate materials. The experiment for numerical demonstration consists of a composite laminate with four plies, where five electrodes are deployed on each side of the plate to inject current and measure the potential. The goal of the experiment is to gain information about the fiber orientations in the composite laminate material from the measured potential. MLDLMC and MLDLSC are applied to efficiently estimating the EIG for a given experiment set-up.

The outline of the paper is as follows. In Section \ref{sec:problemsetting}, we present the EIG criterion and the underlying data model assumption. In Section \ref{sec:assumptions}, we detail the numerical discretization approximation of the EIG. The Laplace-based importance sampling that is used in both of the proposed multilevel methods is given in Section \ref{sec:laplace}. The MLDLMC estimator is presented in Section \ref{sec:meif}. Then, in Section \ref{sec:biasvarwork}, the bias, variance, and the work of DLMCIS is analyzed, and then used in Section \ref{sec:mldlmcparam} when determining suitable values for method parameters in the MLDLMC estimator by minimizing the computational work for a specified accuracy requirement. In Section \ref{sec:compdiscuss}, the work complexity of MLDLMC, with respect to a desired error tolerance, is compared to that of the standard MLMC. The proposed MLDLSC is described in \ref{sec:mldlscsec}. Finally, in Section \ref{sec:eitnum}, we provide a numerical comparison of the computational performances of the two methods, MLDLMC and MLDLSC, for estimating the expected information gain for an EIT experiment.

\section{Problem setting}
\subsection{Bayesian optimal experimental design}\label{sec:problemsetting}
In this work, we consider the data model
\begin{equation}
\label{eq_datamodel}
\bm{Y}(\bm{\theta},\bm{\xi}) = \bm{G}(\bm{\theta}_t,\bm{\xi}) + \Epsilon,
\end{equation}
where $\bm{Y} \myeq \left(\bm{y}_1, \ldots,\bm{y}_i, \ldots, \bm{y}_{N_e}\right)$, $\bm{y}_i \in \mathbb{R}^{q}$ are observed experiment responses, $N_e$ is the number of repeated experiments, $\bm{G}(\bm{\theta}_t,\bm{\xi}) \myeq \bm{g}(\bm{\theta}_t,\bm{\xi})\bm{1}$ with $\bm{1} \myeq (1,\ldots,1)$, $\bm{g}(\bm{\theta}_t,\bm{\xi}) \in \mathbb{R}^q$ is the column vector of forward model outputs, $\bm{\theta}_t \in \mathbb{R}^{d}$ is the \emph{true} parameter, $\bm{\xi} \in \Xi$ is the design parameter, $\Xi$ is the experimental design space, and $\Epsilon \myeq (\bm{\epsilon}_1,\ldots,\bm{\epsilon}_i,\ldots,\bm{\epsilon}_{N_e})$, where $\bm{\epsilon}_i \in \mathbb{R}^q$ are independent and identically distributed (i.i.d.) zero-mean Gaussian errors with the covariance matrix $\bm{\Sigma}_{\bm{\epsilon}}$, that is, $\Epsilon_i \myiid \mathcal{N}(\bm{0},\bm{\Sigma}_{\bm{\epsilon}})$ and the distribution $\pi_\epsilon(\Epsilon)=\prod^{N_e}_{i=1}\pi_{\epsilon_i}(\bm{\epsilon}_i)$.

We consider the case when the parameter $\bm{\theta}_t$ is unknown. To this end, we treat $\bm{\theta}_t$ as a random parameter, $\bm{\theta} \in \Theta$, with prior distribution $\pi(\bm{\theta})$, defined on the space $\Theta \subseteq \mathbb{R}^{d}$.

The goal of Bayesian optimal experimental design is to determine the optimal set-up of an experiment as defined by the design parameter $\bm{\xi}$ for Bayesian inference of $\bm{\theta}_t$. The information gain for a given experimental design, $\bm{\xi}$, is measured by the Kullback-Leibler divergence \cite{kullback1959}, which is based on the Shannon entropy \cite{key12}. The Kullback-Leibler divergence, denoted by $\infdiv{\pi(\bm{\theta} \vert \bm{Y}, \bm{\xi})}{\pi(\bm{\theta})}$, is a distance measure between prior $\pi(\bm{\theta})$ and posterior $\pi(\bm{\theta} \vert \bm{Y}, \bm{\xi})$ pdfs, i.e.,
\begin{equation}
\infdiv{\pi(\bm{\theta} \vert \bm{Y}, \bm{\xi})}{\pi(\bm{\theta})} \myeq \int_{\Theta}{\pi(\bm{\theta} \vert \bm{Y}, \bm{\xi})  \log \left( \frac{\pi(\bm{\theta} \vert \bm{Y},\bm{\xi})}{\pi(\bm{\theta})} \right) \,d\bm{\theta}}. \label{eq_dkl1}
\end{equation}
The larger the value of $D_{\text{KL}}$, the more informative the given experiment is about the unknown parameter $\bm{\theta}_t$. The information gains for different designs, $\bm{\xi}$, are mutually independent of each other. Henceforth, we omit dependences on $\bm{\xi}$ for the sake of conciseness. Since $\bm{Y}$ is not available to us during the design selection, we work with the expected value of $D_{\text{KL}}$,
\begin{align}\label{expecinfgain}
I \myeq& \int_{\mathcal{Y}}{\infdiv{\pi(\bm{\theta} \vert \bm{Y})}{\pi(\bm{\theta})}p(\bm{Y}) \,d\bm{Y}} = \int_{\mathcal{Y}} \int_{\Theta}  \log \left( \frac{\pi(\bm{\theta} \vert \bm{Y})}{\pi(\bm{\theta})} \right) \pi(\bm{\theta} \vert \bm{Y}) \,d\bm{\theta} p(\bm{Y}) \,d\bm{Y} \nonumber \\
\quad =& \int_{\Theta} \int_{\mathcal{Y}} \log \left( \frac{p(\bm{Y} \vert \bm{\theta})}{p(\bm{Y})} \right)  p(\bm{Y} \vert \bm{\theta}) \,d\bm{Y} \pi(\bm{\theta}) \,d\bm{\theta},
\end{align}
which is also known as Expected Information Gain (EIG) \cite{key13}, and this is the design criterion considered in this work for Bayesian optimal experimental design. The latter equality follows from Bayes' rule, and $p(\bm{Y})$ denotes the pdf of $\bm{Y}$ over the support $\mathcal{Y} \myeq \mathbb{R}^{q \times N_e}$. In accordance with the data model \eqref{eq_datamodel}, the likelihood, denoted by $p(\bm{Y}\vert \bm{\theta})$, is
\begin{equation*}
p(\bm{Y} \vert \bm{\theta}) \myeq \text{det}\left(2\pi\bm{\Sigma_\epsilon}\right)^{-\frac{N_e}{2}} \exp \left( -\frac{1}{2}  \sum_{i=1}^{N_e} \left\| \bm{y}_{i} - \bm{g}(\bm{\theta})\right\|^2_{\bm{\Sigma_\epsilon}^{-1}} \right),
\end{equation*}
where the matrix norm is $ \|\bm{x}\|^{2}_{\bm{\Sigma}^{-1}_{\bm{\epsilon}}} = \bm{x}^{T} \bm{\Sigma}^{-1}_{\bm{\epsilon}} \bm{x}$ for a vector $\bm{x}$ and covariance matrix $\bm{\Sigma}_{\bm{\epsilon}}$. For notational convenience, we introduce the conditional expectation, 
\begin{align}\label{eq:Z.def}
\mathcal{Z}(\bm{\theta}) &\myeq \mathbb{E}\left[ f(\bm{Y}, \bm{\theta}) \vert \bm{\theta} \right] \nonumber= \int_{\mathcal{Y}} \log \left( \frac{p(\bm{Y} \vert \bm{\theta})}{p(\bm{Y})} \right)  p(\bm{Y} \vert \bm{\theta}) \,d\bm{Y} \nonumber \\
&= \int_{\mathcal{E}} \log \left( \frac{\pi_\epsilon(\Epsilon)}{\int_{\Theta} \text{det}\left(2\pi\bm{\Sigma_\epsilon}\right)^{-\frac{N_e}{2}} \exp \left( -\frac{1}{2}  \sum\limits_{i=1}^{N_e} \left\| \bm{g}(\bm{\theta}) + \bm{\epsilon}_{i} - \bm{g}(\bm{\theta}^\prime)\right\|^2_{\bm{\Sigma_\epsilon}^{-1}} \right) \pi(\bm{\theta}^\prime) \,d\bm{\theta}^\prime} \right) \pi_\epsilon(\Epsilon) \,d\Epsilon,
\end{align}
where
\begin{equation}
\label{eq:f}
f(\bm{Y}, \bm{\theta}) \myeq \log \left(\frac{p(\bm{Y} \vert \bm{\theta})}{p(\bm{Y})} \right).
\end{equation}
Using \eqref{eq:Z.def}, we formulate the EIG criterion \eqref{expecinfgain} as
\begin{equation}
\label{eq:eigf}
I=\mathbb{E}\left[ \mathcal{Z}(\bm{\theta})\right].
\end{equation}

\subsection{Numerical approximation of expected information gain}
\label{sec:assumptions}
We consider the situation when model $\bm{g}$ needs to be evaluated by a numerical approximation through discretization in space of $\bm{g}$, denoted by $\bm{g}_{\ell}$, with an accuracy controlled by the mesh-element size, denoted by $h_{\ell}>0$. We consider a sequence of such discretization-based approximations, $\{ \bm{g}_\ell \}_{\ell=0}^{\infty}$, with decreasing mesh-element size, i.e., $h_{\ell}<h_{\ell-1}$.
\begin{Assumption}[Convergence properties of $\bm{g}_{\ell}$]\quad\\[3pt]
\label{assump:g}
\noindent \textbf{[Weak convergence]} $\exists C_w>0, \eta_w>0$ such that
\begin{equation}
\Vert \mathbb{E}[\bm{g}_{\ell}(\bm{\theta})-\bm{g}(\bm{\theta})]\Vert_{\bm{\Sigma}^{-1}_{\bm{\epsilon}}} \le C_{w}h_{\ell}^{\eta_w},
\end{equation}
for all $\ell$ as $h_\ell \to 0$.\\
\textbf{[Strong convergence]} For $p \ge 2$, $\exists C_s>0,\eta_s>0$ such that
\begin{equation}
\mathbb{E}[\Vert \bm{g}_{\ell}(\bm{\theta})-\bm{g}_{\ell-1}(\bm{\theta}) \Vert^p_{\bm{\Sigma}^{-1}_{\bm{\epsilon}}}]^{\frac{1}{p}} \le C_{s}h_{\ell}^{\eta_s},
\end{equation}
for all $\ell>0$ as $h_{\ell} \to 0$.\\
\textbf{[Average computational work]} $\exists \gamma>0$ such that
\begin{equation}
\label{eq:gellwork}
W(\bm{g}_{\ell}) \propto h_{\ell}^{-\gamma},
\end{equation}
with respect to $\ell$ as $h_\ell \to 0$, wherein $W(\cdot)$ denotes the average computational work.\\
\end{Assumption}
Assumption \ref{assump:g} states some necessary assumptions on $\bm{g}_{\ell}$ that will be used later on. Furthermore, $\bm{g}_{\ell}$ needs to be twice differentiable with respect to $\bm{\theta}$ and uniformly bounded by some constant independent of $\ell$. Since we have assumed that we only work with the approximations of $\bm{g}$, let us introduce the data model
\begin{equation}
\label{eq:approxdata}
\bm{Y}=\bm{G}_\ell(\bm{\theta})+\Epsilon,
\end{equation}
where $\bm{G}_{\ell}(\bm{\theta}) \myeq \bm{g}_{\ell}(\bm{\theta})\bm{1}$ approximates forward model $\bm{G}(\bm{\theta})$. The approximate EIG criterion \eqref{eq:eigf}, given the approximate data model \eqref{eq:approxdata}, is defined as follows:
\begin{equation}
\label{eq:klfl}
I_{\ell} \myeq \int_{\Theta} \mathcal{Z}_{\ell}(\bm{\theta}) \pi(\bm{\theta}) \,d\bm{\theta},
\end{equation}
where
\begin{align}
\label{eq:Zell}
\mathcal{Z}_{\ell}(\bm{\theta}) &\myeq \int_{\mathcal{Y}} f_{\ell}(\bm{Y},\bm{\theta}) p_{\ell}(\bm{Y} \vert \bm{\theta}) \,d\bm{Y} \nonumber\\
&= \int_{\mathcal{E}} \log \left( \frac{\pi_\epsilon(\Epsilon)}{\int_{\Theta} \text{det}\left(2\pi\bm{\Sigma_\epsilon}\right)^{-\frac{N_e}{2}} \exp \left( -\frac{1}{2}  \sum\limits_{i=1}^{N_e} \left\| \bm{g}_\ell(\bm{\theta}) + \bm{\epsilon}_{i} - \bm{g}_\ell(\bm{\theta}^\prime)\right\|^2_{\bm{\Sigma_\epsilon}^{-1}} \right) \pi(\bm{\theta}^\prime) \,d\bm{\theta}^\prime} \right) \pi_\epsilon(\Epsilon) \,d\Epsilon,
\end{align}

from using $\bm{g}_{\ell}$ in $\mathcal{Z}(\bm{\theta})$, and the log-ratio $f$ is approximated by
\begin{equation}
\label{eq:fell}
f_{\ell}(\bm{Y}, \bm{\theta}) \myeq \log \left(\frac{p_{\ell}(\bm{Y} \vert \bm{\theta})}{p_{\ell}(\bm{Y})} \right).
\end{equation}
We define an approximate likelihood by
\begin{equation}
\label{eq:approxlike}
p_{\ell}(\bm{Y} \vert \bm{\theta}) \myeq \text{det}\left(2\pi\bm{\Sigma_\epsilon} \right)^{-\frac{N_e}{2}} \exp \left( -\frac{1}{2}  \sum_{i=1}^{N_e} \left\| \bm{y}_{i} - \bm{g}_{\ell}(\bm{\theta})\right\|^2_{\bm{\Sigma_\epsilon}^{-1}} \right),
\end{equation}
and we approximate the evidence, for any $\bm{Y}$ satisfying the approximate data model \eqref{eq:approxdata}, as
\begin{equation*}
p_{\ell}(\bm{Y}) \myeq \int_{\Theta} p_{\ell}(\bm{Y} \vert \bm{\theta}^\prime) \pi(\bm{\theta}^\prime) d\bm{\theta}^\prime=\int_{\Theta} \text{det}\left(2\pi\bm{\Sigma_\epsilon}\right)^{-\frac{N_e}{2}} \exp \left( -\frac{1}{2}  \sum_{i=1}^{N_e} \left\| \bm{g}_\ell(\bm{\theta}) + \bm{\epsilon}_{i} - \bm{g}_\ell(\bm{\theta}^\prime)\right\|^2_{\bm{\Sigma_\epsilon}^{-1}} \right) \pi(\bm{\theta}^\prime) \,d\bm{\theta}^\prime.
\end{equation*}

\subsection{Laplace-based importance sampling in the expected information gain}
\label{sec:laplace}
Whenever the posterior distribution, $\pi_{\ell}(\bm{\theta} \vert \bm{Y})=p_{\ell}(\bm{Y} \vert \bm{\theta})\pi(\bm{\theta})/p_{\ell}(\bm{Y})$, can be well approximated by a multivariate normal distribution, we advocate using the Laplace-based importance sampling \cite{BDELT2018,RDP2015}. More specifically, we introduce an importance sampling distribution, denoted by $\tilde{\pi}_{\ell}(\bm{\theta} \vert \bm{Y})$, to compute the approximate evidence $p_{\ell}(\bm{Y})$ as follows:
\begin{align}
\label{eq:evidence}
p_{\ell}(\bm{Y})&=\int_{\Theta} p_{\ell}(\bm{Y} \vert \bm{\theta}^\prime)\pi(\bm{\theta}^\prime) d\bm{\theta}^\prime=\int_{\Theta} p_{\ell}(\bm{Y} \vert \bm{\theta}^\prime)R_{\ell}(\bm{\theta}^\prime;\bm{Y})\tilde{\pi}_{\ell}(\bm{\theta}^\prime \vert \bm{Y}) d\bm{\theta}^\prime \nonumber\\
&\mskip-50mu= \int_{\Theta} \text{det}\left(2\pi\bm{\Sigma_\epsilon}\right)^{-\frac{N_e}{2}} \exp \left( -\frac{1}{2}  \sum_{i=1}^{N_e} \left\| \bm{g}_\ell(\bm{\theta}) + \bm{\epsilon}_{i} - \bm{g}_\ell(\bm{\theta}^\prime)\right\|^2_{\bm{\Sigma_\epsilon}^{-1}} \right) R_\ell (\bm{\theta}^\prime;\bm{G}_\ell(\bm{\theta}) + \Epsilon) \tilde{\pi}_\ell(\bm{\theta}^\prime \vert \bm{G}_\ell(\bm{\theta}) + \Epsilon) \,d\bm{\theta}^\prime,
\end{align}
where the ratio, $R_{\ell}$, is
\begin{equation}
\label{eq:Rell}
R_{\ell}(\bm{\theta};\bm{Y}) \myeq \pi(\bm{\theta})/\tilde{\pi}_{\ell}(\bm{\theta} \vert \bm{Y}).
\end{equation}
The Laplace-based importance sampling measure, $\tilde{\pi}_{\ell}$, is a multivariate normal pdf, denoted \\ by $\mathcal{N}(\hat{\bm{\theta}}_{\ell}(\bm{Y}),\bm{\Sigma}_{\ell}(\hat{\bm{\theta}}_{\ell}(\bm{Y})))$, i.e.,
\begin{equation}
\label{eq:istilde}
\tilde{\pi}_{\ell}(\bm{\theta} \vert \bm{Y}) \myeq \text{det}\left(2\pi\bm{\Sigma}_{\ell}(\hat{\bm{\theta}}_{\ell}(\bm{Y})) \right)^{-\frac{1}{2}} \exp \left( -\frac{1}{2} \left\| \bm{\theta} - \hat{\bm{\theta}}_{\ell}(\bm{Y})\right\|^2_{\bm{\Sigma}^{-1}_{\ell}(\hat{\bm{\theta}}_{\ell}(\bm{Y}))} \right),
\end{equation}
where $\hat{\bm{\theta}}_{\ell}(\bm{Y})$ is the \emph{maximum a posteriori} (MAP) estimate,
\begin{eqnarray}
 \label{fittheta}
 \hat{\bm{\theta}}_{\ell}(\bm{Y})  & \myeq & \underset{\bm{\theta} \in \Theta}{\arg\min} \left[ \frac{1}{2} \sum_{i=1}^{N_e} \left\| \bm{y}_{i} - \bm{g}_{\ell}(\bm{\theta}) \right\|^2_{\bm{\Sigma_\epsilon}^{-1}}-\log (\pi(\bm{\theta})) \right],
\end{eqnarray}
and, as shown in \cite{key15}, the covariance is the inverse Hessian matrix of the negative logarithm of the posterior pdf,
\begin{equation}
\label{eq:sigmahat}
\bm{\Sigma}_{\ell} (\bm{\theta}) \myeq \Big( N_e \bm{J}_{\ell}(\bm{\theta})^T\bm{\Sigma_\epsilon}^{-1} \bm{J}_{\ell}(\bm{\theta})-\nabla_{\bm{\theta}}  \nabla_{\bm{\theta}}  \log (\pi(\bm{\theta})) \Big)^{-1} + \mathcal{O}_\mathbb{P}\left(\frac{1}{\sqrt{N_e}}\right),
\end{equation}
where $\bm{J}_{\ell}(\bm{\theta})\myeq-\nabla_{\bm{\theta}} \bm{g}_{\ell}(\bm{\theta})$. Note that $\hat{\bm{\theta}}_{\ell}$ depends on the data $\bm{Y}$. Moreover, \eqref{eq:sigmahat} says that the larger the number of repetitive experiments $N_e$, the more accurately we can approximate the covariance $\bm{\Sigma}_{\ell}(\hat{\bm{\theta}}_{\ell}(\bm{Y}))$ of the importance-sampling pdf, $\tilde{\pi}_{\ell}$.

\section{Multilevel Double Loop Monte Carlo}
The standard MLMC \cite{key14,key25} has been widely applied and extended to various problems \cite{G2015}. The idea behind multilevel methods is to not only compute the expectation of the quantity of interest using $\bm{g}_{\ell}$ for a fine mesh-element size $h_{\ell}$, but instead reducing the computational work complexity by distributing the computations over a sequence of $L+1$ mesh-element sizes, $\{ h_{\ell} \}^L_{\ell=0}$, from coarse to fine meshes, and then combine the results. Multilevel methods distribute the computational workload such that the majority of the model evaluations are on the coarser meshes. The classical choice of decreasing sequence is
\begin{equation}
\label{eq:hell}
h_{\ell} \myeq \beta^{-\ell}h_0, \quad \text{for some}\,\, \beta \in \mathbb{N}^{+};
\end{equation}
typically $\beta=2$, i.e., to progressively halve the size with increasing $\ell$, and with $h_0$ being the coarsest mesh-element size considered.

\subsection{Multilevel Double Loop Monte Carlo (MLDLMC) estimator}
\label{sec:meif}
The approximate EIG \eqref{eq:klfl} at level $L$ can be written as the telescopic sum
\begin{equation}\label{eq:telescop}
I_{L}=\sum_{\ell=0}^{L}{\mathbb{E}\left[\Delta\mathcal{Z}_{\ell}(\bm{\theta})\right]},
\end{equation}
where the index $\ell$ is here referred to as the ``level,'' and
\begin{align}
  \label{eq:deltaZ}
   \Delta[\mathcal{Z}_{\ell}(\bm{\theta})] & \myeq
    \begin{cases}
     \mathcal{Z}_{\ell}(\bm{\theta})-\mathcal{Z}_{\ell-1}(\bm{\theta}) & \text{if $\ell>0$,}\\
     \mathcal{Z}_{\ell}(\bm{\theta}) & \text{if $\ell=0$}.
    \end{cases}
\end{align}

The function $\mathcal{Z}_{\ell}$ depends on $f_{\ell}$, \eqref{eq:Zell}, and, in turn, $f_{\ell}$ depends on the approximate evidence $p_{\ell}$ \eqref{eq:evidence}. We estimate the approximate evidence by MC sampling with the Laplace-based importance sampling described in Section \ref{sec:laplace}, i.e., for each $\bm{\theta} \myiid \pi(\bm{\theta})$ and $\Epsilon_i \myiid \mathcal{N}(\bm{0},\bm{\Sigma}_{\bm{\epsilon}})$, sample $\bm{Y}=\bm{G}_\ell(\bm{\theta})+\Epsilon$, then compute the approximate evidence as follows:
\begin{equation}
\label{eq:pM}
\hat{p}_{\ell}(\bm{Y},\bm{\theta};\{\bm{\theta}_m\}) \myeq \frac{1}{M_{\ell}}\sum^{M_{\ell}}_{m=1} p_{\ell}(\bm{Y} \vert \bm{\theta}_{m})R_{\ell}(\bm{\theta};\bm{Y}) \approx p_{\ell}(\bm{Y}),
\end{equation}
where $\bm{\theta}_m \myiid \tilde{\pi}_\ell(\bm{\theta}|\bm{Y})=\mathcal{N}(\hat{\bm{\theta}}_\ell(\bm{Y}),\bm{\Sigma}(\hat{\bm{\theta}}_{\ell}(\bm{Y})))$ from expressions \eqref{eq:istilde} and \eqref{fittheta}, the number of samples is denoted by $M_{\ell}$, and $R_\ell(\bm{\theta};\bm{Y})$ is given in \eqref{eq:Rell}. As in standard MLMC \cite{key14,G2015}, we apply sample averaging to the $L+1$ telescopic, conditional expectation, differences, to obtain an MLDLMC estimator for the EIG criterion defined in \eqref{expecinfgain}. Let us introduce 
\begin{equation}
\label{eq:fhat}
\hat{f}_{\ell}(\bm{Y},\bm{\theta};\{\bm{\theta}_m\}) \myeq \log \left(\frac{p_{\ell}(\bm{Y} \vert \bm{\theta})}{\hat{p}_{\ell}(\bm{Y},\bm{\theta};\{\bm{\theta}_m\})} \right).
\end{equation}
Then, by using the approximate EIG \eqref{eq:telescop} with data following \eqref{eq:approxdata} for each level $\ell$ in \eqref{eq:telescop}, the MLDLMC estimator for the EIG criterion \eqref{expecinfgain} reads
\begin{eqnarray}
\label{eq:mldl}
\mathcal{I}_{\text{MLDLMC}} &\myeq& \frac{1}{N_0}\sum^{N_0}_{n=1} \hat{f}_{0}(\bm{Y}^{(0)}_{0,n},\bm{\theta}_{0,n};\{ \bm{\theta}_{0,n,m}\}^{M_0}_{m=1})\nonumber\\
&+&\sum_{\ell=1}^{L} \frac{1}{N_{\ell}} \sum^{N_{\ell}}_{n=1} \left[ \hat{f}_{\ell}(\bm{Y}^{(\ell)}_{\ell,n},\bm{\theta}_{\ell,n};\{\bm{\theta}_{\ell,n,m}\}^{M_\ell}_{m=1})-\hat{f}_{\ell-1}(\bm{Y}^{(\ell-1)}_{\ell,n},\bm{\theta}_{\ell,n};\{\bm{\theta}_{\ell,n,m}\}^{M_{\ell-1}}_{m=1}) \right],
\end{eqnarray}
where $\bm{Y}^{(k)}_{\ell,n} = \bm{G}_k(\bm{\theta}_{\ell,n}) + \Epsilon_{\ell,n} \myiid p_{k}(\bm{Y} \vert \bm{\theta}_{\ell,n})$, $\bm{\theta}_{\ell,n} \myiid \pi(\bm{\theta})$, $\bm{\theta}_{k,n,m} \myiid \tilde{\pi}_{k}(\bm{\theta} \vert \bm{Y}^{(k)}_{\ell,n}) = \mathcal{N}(\hat{\bm{\theta}}(\bm{Y}^{(k)}_{\ell,n}),\bm{\Sigma}(\hat{\bm{\theta}}(\bm{Y}^{(k)}_{\ell,n})))$, and $\{\bm{\theta}_{\ell,n,m}\}^{M_{\ell-1}}_{m=1} \subseteq \{\bm{\theta}_{\ell,n,m}\}^{M_{\ell}}_{m=1}$. Here the superscript of $\bm{Y}^{(k)}$ implies that the data $\bm{Y}$ depends on $p_k$, as defined in \eqref{eq:approxlike}.

\begin{Remark}[Choice of $\hat{\bm{\theta}}_n$]
As shown in \cite{BDELT2018,RDP2015}, a Laplace-based importance sampling centered on the MAP estimate $\hat{\bm{\theta}}_n \myeq \hat{\bm{\theta}}(\bm{Y}_n)$ can drastically reduce the number of inner samples. In fact, it was demonstrated in \cite{BDELT2018} that even a few samples can be sufficient for moderate error tolerances, which is equivalent to using the Laplace method as in \cite{key15} but centered on $\hat{\bm{\theta}}_n$ instead of $\bm{\theta}_n$, where $\bm{\theta}_n$ is the parameter that is used to generate the data $\bm{Y}_n$.
To estimate $\hat{\bm{\theta}}_n$, we require additional evaluations of the forward model for each outer sample. The search for $\hat{\bm{\theta}}_n$ by solving the optimization problem \eqref{fittheta} is substantially reduced when initialized at $\bm{\theta}_n$. As mentioned above, an alternative approach is to center the new measure on $\bm{\theta}_n$, but this is a less accurate approximation because the discrepancy between $\bm{\theta}_n$ and the MAP estimate $\hat{\bm{\theta}}_n$ may be large, risking underflow, which was discussed in detail in \cite{BDELT2018}.
\end{Remark}

\subsection{Bias, variance and work analysis}
\label{sec:biasvarwork}
The bias, variance and computational work of the MLDLMC estimator, \eqref{eq:mldl}, need to be analyzed. MLDLMC is a consistent estimator, i.e., the bias goes to zero asymptotically, and its bias can be bounded from above by, as $L \to \infty$ and $M_L \to \infty$,
\begin{equation}
\label{eq:biasdlmc}
\vert I-\mathbb{E}\left[ \mathcal{I}_{\text{MLDLMC}} \right] \vert \lessapprox C_2h_L^{\eta_w}+C_1M_{L}^{-1},
\end{equation}
with
\begin{equation}
C_{1} = \frac{1}{2}\mathbb{E}\left[\mathbb{V} \left[ \frac{ p(\bm{Y} \vert \bm{\theta})}{p(\bm{Y})} \vert \bm{Y}\right] \right],
\end{equation}
and $C_{2}$ being the constant for the weak convergence of $I_{\ell}$. The upper bound for the bias \eqref{eq:biasdlmc} is from the bias result in Proposition 1 \cite{BDELT2018} for the DLMC estimator with mesh-element size $h_{L}$. The variance of the MLDLMC estimator is
\begin{eqnarray}
\mathbb{V}\left[ \mathcal{I}_{\text{MLDLMC}} \right]=\frac{V_0}{N_0}+\sum^{L}_{\ell=1} \dfrac{V_{\ell}}{N_\ell},
\end{eqnarray}
since the samples for each level are mutually independent to the those of the other levels. Here $V_0$ denotes the variance of the DLMC estimator at the coarsest level ($\ell=0$),
\begin{equation}
V_{0} \myeq \mathbb{V}\left[ \hat{f}_{0}(\bm{Y}^{(0)}_{0,n},\bm{\theta}_{0,n};\{ \bm{\theta}_{0,n,m}\}^{M_0}_{m=1}) \right],
\end{equation}
and, as shown in Proposition 1 \cite{BDELT2018}, as $M_0 \to \infty$ the variance $V_0$ behaves as
\begin{equation}
\label{eq:V0est}
V_{0} \approx C_3+\frac{C_4}{M_0},
\end{equation}
for some constants $C_3,C_4>0$, where $C_4 \ll C_3$ due to the Laplace-based importance sampling \cite{BDELT2018}. The expectation and the variance of the telescopic, conditional expectation, difference at level $\ell$ are defined by
\begin{equation}\label{eq:E_ell}
E_{\ell} \myeq \mathbb{E}\left[ \hat{f}_{\ell}(\bm{Y}^{(\ell)}_{\ell,n},\bm{\theta}_{\ell,n};\{\bm{\theta}_{\ell,n,m}\}^{M_\ell}_{m=1})-\hat{f}_{\ell-1}(\bm{Y}^{(\ell-1)}_{\ell,n},\bm{\theta}_{\ell,n};\{\bm{\theta}_{\ell,n,m}\}^{M_{\ell-1}}_{m=1}) \right] \ \text{for} \ \ell>0,
\end{equation}
and
\begin{equation}\label{eq:V_ell}
V_{\ell} \myeq \mathbb{V}\left[ \hat{f}_{\ell}(\bm{Y}^{(\ell)}_{\ell,n},\bm{\theta}_{\ell,n};\{\bm{\theta}_{\ell,n,m}\}^{M_\ell}_{m=1})-\hat{f}_{\ell-1}(\bm{Y}^{(\ell-1)}_{\ell,n},\bm{\theta}_{\ell,n};\{\bm{\theta}_{\ell,n,m}\}^{M_{\ell-1}}_{m=1}) \right], \ \text{for} \ \ell>0,
\end{equation}
respectively. Theorem \ref{lem:var} provides an asymptotic upper bound to $V_\ell$ for $\ell>0$.
\begin{theorem}\label{lem:var} Given Assumption \ref{assump:g}, then for $\ell>0$,
\begin{equation}
\label{eq:Vlest}
V_{\ell} \lessapprox  \left[ M_\ell \left(\dfrac{1}{M_\ell} - \dfrac{1}{M_{\ell-1}}\right)^2 + \dfrac{h^{2\eta_s}_\ell}{M_{\ell-1}} + \dfrac{M_\ell - M_{\ell-1}}{M_{\ell-1}^2} \right] + h_\ell^{2\eta_w},
\end{equation}
as $M_0 \to \infty$.
\end{theorem}
\begin{proof}
The goal is to find an upper bound for the variance of the differences in the MLDLMC estimator \eqref{eq:mldl} denoted by $V_{\ell}$, i.e.,
\begin{equation}
\label{eq:Velbound}
V_{\ell} \myeq \mathbb{V}\left[ \hat{f}_{\ell}(\bm{Y},\bm{\theta};\{\bm{\theta}_{\ell,n,m}\}^{M_\ell}_{m=1})-\hat{f}_{\ell-1}(\bm{Y},\bm{\theta};\{\bm{\theta}_{\ell,n,m}\}^{M_{\ell-1}}_{m=1}) \right] = \mathbb{V}\left[ \log\left(\dfrac{\hat{p}_\ell(\bm{Y},\bm{\theta};\{ \bm{\theta_{m}} \}_{m=1}^{M_\ell})}{\hat{p}_{\ell-1}(\bm{Y},\bm{\theta};\{ \bm{\theta_{m}} \}_{m=1}^{M_{\ell-1}})}\right) \right].
\end{equation}
First, consider the ratio between the approximate evidences evaluated for $\ell$ and $\ell-1$,
\begin{equation}\label{eq:evidence.ratio}
\log\left(\dfrac{\hat{p}_\ell(\bm{Y},\bm{\theta};\{ \bm{\theta_{m}} \}_{m=1}^{M_\ell})}{\hat{p}_{\ell-1}(\bm{Y},\bm{\theta};\{ \bm{\theta_{m}} \}_{m=1}^{M_{\ell-1}})}\right),
\end{equation}
where
\begin{equation}\label{eq:Xell}
\hat{p}_\ell(\bm{Y},\bm{\theta};\{ \bm{\theta}_{m} \}_{m=1}^{M_\ell}) \propto \dfrac{1}{M_\ell} \sum_{m=1}^{M_\ell} \underbrace{\exp\left( -\frac{1}{2}  \sum_{i=1}^{N_e} \left\| \bm{y}^{(\ell)}_{i,\ell,n} - \bm{g}_{i,\ell}(\bm{\theta}_{\ell,n,m})\right\|^2_{\bm{\Sigma_\epsilon}^{-1}} \right)R_{\ell,n,m}}_{\myeq X^{(\ell)}_{\ell,n,m}},
\end{equation}
and denote $\mu^{(\ell)}_{\ell,n} \myeq \mathbb{E}[X^{(\ell)}_{\ell,n,m}|\bm{\theta}_{\ell,n},\bm{\epsilon}_{\ell,n}]$ and $R_{\ell,n,m} \myeq R_{\ell}(\bm{\theta}_{\ell,n,m};\bm{Y}^{(\ell)}_{\ell,n})$.
Similarly, consider
\begin{equation}\label{eq:Xellm1}
\hat{p}_{\ell-1}(\bm{Y},\bm{\theta};\{ \bm{\theta}_{m} \}_{m=1}^{M_{\ell-1}}) \propto \dfrac{1}{M_{\ell-1}} \sum_{m=1}^{M_\ell} \underbrace{\exp\left( -\frac{1}{2}  \sum_{i=1}^{N_e} \left\| \bm{y}^{(\ell-1)}_{i,\ell,n} - \bm{g}_{i,\ell}(\bm{\theta}_{\ell,n,m})\right\|^2_{\bm{\Sigma_\epsilon}^{-1}}\right) R_{\ell,n,m}}_{\myeq X^{(\ell-1)}_{\ell,n,m}},
\end{equation}
and denote $\mu^{(\ell-1)}_{\ell,n} \myeq \mathbb{E}[X^{(\ell-1)}_{\ell,n,m}|\bm{\theta}_{\ell,n},\bm{\epsilon}_{\ell,n}]$. The approximate evidence ratio \eqref{eq:evidence.ratio} is then rewritten as
\begin{align}\label{eq:evidence.ratio.specialized}
\log\left(\dfrac{\hat{p}_\ell(\bm{Y},\bm{\theta};\{ \bm{\theta}_{m} \}_{m=1}^{M_\ell})}{\hat{p}_{\ell-1}(\bm{Y},\bm{\theta};\{ \bm{\theta}_{m} \}_{m=1}^{M_{\ell-1}})}\right)=& \log\left[\left(\dfrac{\mu^{(\ell)}_{\ell,n}}{\mu^{(\ell-1)}_{\ell,n}}-1\right)+1\right] + \log\left[ \dfrac{\dfrac{1}{M_\ell}\sum_{m=1}^{M_\ell} \left( X^{(\ell)}_{\ell,n,m} - \mu^{(\ell)}_{\ell,n}\right)}{\mu^{(\ell)}_{\ell,n}} + 1 \right] \nonumber \\
&- \log\left[ \dfrac{\dfrac{1}{M_{\ell-1}}\sum_{m=1}^{M_{\ell-1}} \left( X^{(\ell-1)}_{\ell,n,m} - \mu^{(\ell-1)}_{\ell,n}\right)}{\mu^{(\ell-1)}_{\ell,n}} + 1 \right].
\end{align}
We will use that $M_\ell \to \infty$ for all $\ell$ since $M_0 \to \infty$ and $M_{\ell} \ge M_{\ell-1}$ for $\ell>0$. Then, by using Taylor expansion, $\log(x+1) \sim x$ in a neighborhood of zero, we are led to approximate \eqref{eq:evidence.ratio.specialized} by
\begin{align}\label{eq:evidence.ratio.approx}
\log\left(\dfrac{\hat{p}_\ell(\bm{Y},\bm{\theta};\{ \bm{\theta}_{m} \}_{M_\ell})}{\hat{p}_{\ell-1}(\bm{Y},\bm{\theta};\{ \bm{\theta}_{m} \}_{M_{\ell-1}})}\right) -& \left(\dfrac{\mu^{(\ell)}_{n}}{\mu^{(\ell-1)}_{\ell,n}}-1\right) \nonumber \\
\sim & \dfrac{1}{\mu^{(\ell)}_{\ell,n} M_\ell} \sum_{m=1}^{M_\ell} \left( X^{(\ell)}_{\ell,n,m} - \mu^{(\ell)}_{\ell,n}\right) - \dfrac{1}{\mu^{(\ell-1)}_{\ell,n} M_{\ell-1}} \sum_{m=1}^{M_{\ell-1}} \left( X^{(\ell-1)}_{\ell,n,m} - \mu^{(\ell-1)}_{\ell,n}\right), \nonumber\\
=&\left(\dfrac{1}{\mu^{(\ell)}_{\ell,n} M_\ell} - \dfrac{1}{\mu^{(\ell-1)}_{\ell,n} M_{\ell-1}}\right) \sum_{m=1}^{M_\ell} \left( X^{(\ell)}_{\ell,n,m} - \mu^{(\ell)}_{\ell,n}\right) \nonumber \\
&+ \dfrac{1}{\mu^{(\ell-1)}_{\ell,n} M_{\ell-1}} \Bigg\{ \sum_{m=1}^{M_{\ell-1}} \left[\left(X^{(\ell)}_{\ell,n,m} - X^{(\ell-1)}_{\ell,n,m}\right) - \left(\mu^{(\ell)}_{\ell,n} - \mu^{(\ell-1)}_{\ell,n}\right)\right] \nonumber \\
&+ \sum_{m=M_{\ell-1}+1}^{M_{\ell}} \left( X^{(\ell)}_{\ell,n,m} - \mu^{(\ell)}_{\ell,n}\right) \Bigg\},
\end{align}
where we denote the first term as $A$ and the sum of the last two as $B$. Thus,
\begin{equation}\label{eq:var.sum.approx}
\mathbb{V}[A+B] = \mathbb{V}[A] + \mathbb{V}[B] + 2\mathbb{C}[A,B] \le 2 (\mathbb{V}[A] + \mathbb{V}[B]).
\end{equation}
Next, denoting for any random variable $Z$ conditional expectations and variances by
\[
\mathbb{E}_{\ell,n}\left[Z\right] \myeq
\mathbb{E}\left[Z|\bm{\theta}_{\ell,n},\bm{\epsilon}_{\ell,n}\right] , \,\,
\mathbb{V}_{\ell,n}\left[Z\right] \myeq
\mathbb{V}\left[Z|\bm{\theta}_{\ell,n},\bm{\epsilon}_{\ell,n}\right],
\]
we estimate the variance of expression \eqref{eq:evidence.ratio.approx} by combining it with expression \eqref{eq:var.sum.approx}, we arrive at the following conditional variance estimate:
\begin{align}
\mathbb{V}_{\ell,n}\left[\log\left(\dfrac{\hat{p}_\ell(\bm{Y},\bm{\theta};\{ \bm{\theta_{m}} \}^{M_\ell}_{m=1})}{\hat{p}_{\ell-1}(\bm{Y},\bm{\theta};\{ \bm{\theta_{m}} \}^{M_{\ell-1}}_{m=1})}\right)\right] \le & 2 \left(\dfrac{1}{\mu^{(\ell)}_{\ell,n} M_\ell} - \dfrac{1}{\mu^{(\ell-1)}_{\ell,n} M_{\ell-1}}\right)^2 M_\ell \mathbb{V}_{\ell,n}\left[X^{(\ell)}_{\ell,n,m} - \mu^{(\ell)}_{\ell,n}\right] \nonumber\\
&+ \dfrac{4}{(\mu^{(\ell-1)}_{\ell,n} M_{\ell-1})^2} \Bigg\{ M_{\ell-1}\mathbb{V}_{\ell,n}\left[X^{(\ell)}_{\ell,n,m} - X^{(\ell-1)}_{\ell,n,m}\right] \nonumber \\
&+ (M_\ell - M_{\ell-1}) \mathbb{V}_{\ell,n}\left[X^{(\ell)}_{\ell,n,m} - \mu^{(\ell)}_{\ell,n}\right] \Bigg\}.
\end{align}
Now, to estimate $\mathbb{V}_{\ell,n}\left[X^{(\ell)}_{\ell,n,m} - X^{(\ell-1)}_{\ell,n,m}\right] $ above, we consider the difference
\[\mathcal{E} \myeq \left(X^{(\ell)}_{\ell,n,m} - X^{(\ell-1)}_{\ell,n,m}\right)R_{\ell,n,m} = \left(\exp(T_\ell) - \exp(T_{\ell-1})\right)R_{\ell,n,m},\]
where $T_\ell \myeq X_{\ell,n,m}^{(\ell)}$ and similarly $T_{\ell-1} \myeq X^{(\ell-1)}_{\ell,n,m}$. The difference $\mathcal{E}$ can be rewritten as
\begin{equation}\label{eq:diff.exp.ell}
\mathcal{E} = (\exp(T_\ell-T_{\ell-1})-1)\exp(T_{\ell-1})R_{\ell,n,m}.
\end{equation}
Then, by a Taylor expansion $\exp(x) - 1 \sim x$ in a neighborhood of zero, we are led to approximate \eqref{eq:diff.exp.ell} by
\begin{equation}\label{eq:diff.exp.ell.approx}
\mathcal{E} \sim (T_\ell-T_{\ell-1})\exp(T_{\ell-1})R_{\ell,n,m}.
\end{equation}
Thus, the variance of $\mathcal{E}$ can be bounded as
\begin{equation}
\mathbb{V}_{\ell,n}[\mathcal{E}] \lesssim \mathbb{E}[(T_\ell-T_{\ell-1})^2\exp(2T_{\ell-1})R_{\ell,n,m}^2].
\end{equation}
Insert in the difference $T_\ell-T_{\ell-1}$ the definitions of $X^{(\ell)}_{\ell,n,m}$ in \eqref{eq:Xell} and $X^{(\ell)}_{\ell-1,n,m}$ in \eqref{eq:Xellm1}, and we obtain
\begin{equation}
T_\ell-T_{\ell-1} = -\frac{1}{2} \sum_{i=1}^{N_e} \underbrace{\left( \left\| \bm{y}^{(\ell)}_{i,\ell,n} - \bm{g}_{i,\ell}(\bm{\theta}_{\ell,n,m})\right\|^2_{\bm{\Sigma_\epsilon}^{-1}} - \left\| \bm{y}^{(\ell-1)}_{i,\ell,n} - \bm{g}_{i,\ell-1}(\bm{\theta}_{\ell,n,m})\right\|^2_{\bm{\Sigma_\epsilon}^{-1}} \right)}_{F},
\end{equation}
where $F$ can be written as the following inner product in the $\bm{\Sigma_\epsilon}^{-1}$ norm,
\begin{multline}
F=\Bigg(\overbrace{\vphantom{\bm{y}^{(\ell)}_{i,\ell,n}}\bm{g}_{\ell}(\bm{\theta}_{\ell,n}) - \bm{g}_{\ell-1}(\bm{\theta}_{\ell,n}) - (\bm{g}_{\ell}(\bm{\theta}_{\ell,n,m}) - \bm{g}_{\ell-1}(\bm{\theta}_{\ell,n,m}))}^{\ominus \bm{g}_\ell},
\\ \overbrace{\vphantom{\bm{y}^{(\ell)}_{i,\ell,n}}\bm{g}_{\ell}(\bm{\theta}_{\ell,n})+\bm{g}_{\ell-1}(\bm{\theta}_{\ell,n})-(\bm{g}_{\ell}(\bm{\theta}_{\ell,n,m}-\bm{g}_{\ell-1}(\bm{\theta}_{\ell,n,m}))+2\epsilon_{\ell,n}}^{\oplus \bm{g}_\ell} \Bigg)_{\bm{\Sigma_\epsilon}^{-1}}.
\end{multline}
Also, using Cauchy--Schwartz, we bound $F^2$ by
\begin{equation}
F^2 \le \left\| \ominus\bm{g}_\ell \right\|^2_{\bm{\Sigma_\epsilon}^{-1}} \left\| \oplus\bm{g}_\ell \right\|^2_{\bm{\Sigma_\epsilon}^{-1}}.
\end{equation}
Next, to bound the variance of the approximation of $\mathcal{E}$ in \eqref{eq:diff.exp.ell.approx}, for some suitable $0<p,q$ such that $1/p+1/q=1$,
we use H\"{o}lder inequality yielding
\begin{eqnarray}
\label{eq:holder}
\mathbb{V}_{\ell,n}[e] & \lesssim \mathbb{E}_{\ell,n}[\left\| \ominus\bm{g}_\ell \right\|^2_{\bm{\Sigma_\epsilon}^{-1}} \left\| \oplus\bm{g}_\ell \right\|^2_{\bm{\Sigma_\epsilon}^{-1}} \exp(2 T_{\ell-1})R_{\ell,n,m}^2], \nonumber \\
& \lesssim \mathbb{E}_{\ell,n}[\left\| \ominus\bm{g}_\ell \right\|^{2p}_{\bm{\Sigma_\epsilon}^{-1}}]^{1/p} \,
 \mathbb{E}_{\ell,n}[\left\| \oplus\bm{g}_\ell \right\|^{2q}_{\bm{\Sigma_\epsilon}^{-1}} \exp(2q T_{\ell-1})R_{\ell,n,m}^{2q}]^{1/q}, \nonumber \\
& \lesssim \, h_\ell^{2\eta_s},
\end{eqnarray}
and the constants in the inequalities \eqref{eq:holder} are integrable and depends on $\bm{\theta}_{\ell,n}$, $\bm{Y}^{\ell}_{\ell,n}$ and $\bm{Y}^{\ell-1}_{\ell,n}$.
Thus, the variance of the ratio between the evidences in \eqref{eq:evidence.ratio.approx} is bounded by
\begin{align}\label{eq:variance.estimate}
V_{\ell,n}= \mathbb{V}_{\ell,n}\left[\log\left(\dfrac{\hat{p}_\ell(\bm{Y},\bm{\theta};\{ \bm{\theta_{m}} \}_{m=1}^{M_\ell})}{\hat{p}_{\ell-1}(\bm{Y},\bm{\theta};\{ \bm{\theta_{m}} \}_{m=1}^{M_{\ell-1}})}\right)\right] \lessapprox & M_\ell \left(\dfrac{1}{\mu^{(\ell)}_{\ell,n} M_\ell} - \dfrac{1}{\mu^{(\ell-1)}_{\ell,n} M_{\ell-1}}\right)^2 + \dfrac{h^{2\eta_s}_\ell}{\mu^{(\ell-1)}_{\ell,n} M_{\ell-1}} \nonumber \\
& + \dfrac{M_\ell - M_{\ell-1}}{(\mu^{(\ell-1)}_{\ell,n} M_{\ell-1})^2}, \nonumber \\
\lessapprox & \dfrac{1}{\mu^{(\ell)}_{\ell,n}} \left[ M_\ell \left(\dfrac{1}{M_\ell} - \dfrac{1}{M_{\ell-1}}\right)^2 + \dfrac{h^{2\eta_s}_\ell}{M_{\ell-1}} + \dfrac{M_\ell - M_{\ell-1}}{M_{\ell-1}^2} \right].
\end{align}
To conclude on the optimal work, we need to estimate the total variance corresponding to level $\ell$, namely
\[
V_{\ell} = \mathbb{E}\left[\mathbb{V}_{\ell,n}\left[\log\left(\dfrac{\hat{p}_\ell(\bm{Y},\bm{\theta};\{ \bm{\theta_{m}} \}_{M_\ell})}{\hat{p}_{\ell-1}(\bm{Y},\bm{\theta};\{ \bm{\theta_{m}} \}_{M_{\ell-1}})}\right)\right] \right] + \mathbb{V}\left[\mathbb{E}_{\ell,n}\left[\log\left(\dfrac{\hat{p}_\ell(\bm{Y},\bm{\theta};\{ \bm{\theta_{m}} \}_{M_\ell})}{\hat{p}_{\ell-1}(\bm{Y},\bm{\theta};\{ \bm{\theta_{m}} \}_{M_{\ell-1}})}\right)\right] \right].
\]
Combining \eqref{eq:variance.estimate} with
\[
\left|\mathbb{E}_{\ell,n}\left[\log\left(\dfrac{\hat{p}_\ell(\bm{Y},\bm{\theta};\{ \bm{\theta_{m}} \}_{M_\ell})}{\hat{p}_{\ell-1}(\bm{Y},\bm{\theta};\{ \bm{\theta_{m}} \}_{M_{\ell-1}})}\right)\right]\right| \sim \left|\dfrac{\mu^{(\ell)}_{\ell,n}}{\mu^{(\ell-1)}_{\ell,n}}-1\right| \sim h_\ell^{\eta_w}
\]
yields
\[
V_{\ell} \lessapprox  \left[ M_\ell \left(\dfrac{1}{M_\ell} - \dfrac{1}{M_{\ell-1}}\right)^2 + \dfrac{h^{2\eta_s}_\ell}{M_{\ell-1}} + \dfrac{M_\ell - M_{\ell-1}}{M_{\ell-1}^2} \right] + h_\ell^{2\eta_w}.
\]
This concludes the proof.
\end{proof}
Theorem \ref{lem:var} is an important result as it shows the variance decay per level $\ell$, as $M_{\ell} \to \infty$, which will later on be used when estimating the statistical error of the MLDLMC estimator. For comparison purposes, we note the same result for the standard MLMC for estimating the expected value of $\Vert \bm{g} \Vert_{\bm{\Sigma}^{-1}_{\bm{\epsilon}}}$ is given by $V_{\ell} \lessapprox h_{\ell}^{2\eta_s}$ for $\ell>0$ \cite{G2015}. This comparison highlights the extra layer of complexity for multilevel techniques when the quantities of interest are more challenging in structure, such as the nested expectation form of the EIG criterion.
The average computational work of the MLDLMC estimator \eqref{eq:mldl} is modeled as
\begin{equation}
\label{eq:workMLDL}
W(\mathcal{I}_{\text{MLDLMC}}) \propto \sum^{L}_{\ell=0} N_{\ell}M_{\ell}W(\bm{g}_{\ell}),
\end{equation}
where $W(\bm{g}_{\ell})$ is the average work of a single evaluation of $\bm{g}_{\ell}$ and is modeled as $W(\bm{g}_{\ell}) \propto h_{\ell}^{-\gamma}$, see \eqref{eq:gellwork} in Assumption \ref{assump:g}. Note that work model \eqref{eq:workMLDL} exploits that $W(\hat{f}_{\ell}-\hat{f}_{\ell-1}) \propto M_{\ell}W(\bm{g}_{\ell})$.
\subsection{Choice of MLDLMC parameters}
\label{sec:mldlmcparam}
Following the approach in \cite{key56}, we select the values of the MLDLMC parameters, $L$, $\{ M_{\ell} \}^{L}_{\ell=0}$ and $\{ N_{\ell} \}_{\ell=0}^L$, for a random estimator $\mathcal{I}$ (short for $\mathcal{I}_{\text{MLDLMC}}$) that minimizes the average computational work such that the absolute value of the error, $\vert I-\mathcal{I} \vert$, is less than or equal to a desired error tolerance $\hbox{TOL}>0$ with probability $1-\alpha$, i.e.,
\begin{equation}
\label{eq:probframe}
\mathbb{P}\left(\vert I-\mathcal{I} \vert \leq \hbox{TOL} \right) \geq 1-\alpha,
\end{equation}
where $0<\alpha<1$ and, typically, $\alpha \ll 1$. A solution to the above optimization problem can be found by solving the problem below, where we split the total error into a bias component and a statistical error:
\begin{equation*}
  \lvert I - \mathcal{I} \rvert \leq \lvert I - \mathbb{E}\left[\mathcal{I}\right] \rvert + \lvert \mathbb{E} \left[\mathcal{I}\right] - \mathcal{I} \rvert.
\end{equation*}
Then, we minimize the average work such that the constraints
\begin{eqnarray}
\lvert I - \mathbb{E}\left[\mathcal{I}\right] \rvert &\leq& (1-\kappa) \hbox{TOL} \quad \text{and} \label{eq:bias} \\
\lvert \mathbb{E} \left[\mathcal{I}\right] - \mathcal{I} \rvert &\leq& \kappa \hbox{TOL} \label{eq:vareq}
\end{eqnarray}
hold for a balancing parameter $0 < \kappa < 1$. The constraint \eqref{eq:bias} is the bias constraint, and thee second constraint \eqref{eq:vareq} is a statistical constraint. The constraint \eqref{eq:vareq}, imposed on the statistical error, must hold with probability of $1-\alpha$. From a Central Limit Theorem (Theorem 1.1 \cite{HK2018}; Lemma 7.1 \cite{key56}) for normalized MLMC estimators, if $\eta_s>\gamma$, then
\begin{equation*}
\dfrac{\mathcal{I}-I_L}{\sqrt{\mathbb{V}[\mathcal{I}]}} \rightharpoonup \mathcal{N}(0,1), \quad \text{as } \hbox{TOL} \to 0,
\end{equation*}
where $\mathcal{N}(0,1)$ is a standard normal random variable, and $\rightharpoonup$ denotes convergence in distribution. Therefore, the statistical error constraint \eqref{eq:vareq} is approximated by a variance constraint, which is easier to handle numerically, i.e.,
\begin{eqnarray}\label{eq:statconst}
\mathbb{V}\left[\mathcal{I}\right] \leq \left( \frac{\kappa \hbox{TOL}}{C_\alpha}\right)^2,
\end{eqnarray}
where $C_\alpha = \Phi^{-1}(1-\frac{\alpha}{2})$ and $\Phi^{-1}(\cdot)$ is the inverse cumulative distribution function (cdf) of the standard normal distribution.
By using the results for the bias, variance and work given in Section \ref{sec:biasvarwork}, we can state an optimization problem for finding the method parameters of MLDLMC subject to \eqref{eq:probframe}:
\begin{eqnarray*}
 \argmin_{(N_{\ell},M_{\ell},L,\kappa)} && \sum^L_{\ell=0} N_{\ell}M_{\ell}h_{\ell}^{-\gamma} \\
\hbox{subject to}\;\;\; && \left\{
\begin{array}{ll}
C_{2}h_L^{\eta_w}+C_{1}M_{L}^{-1} \leq (1-\kappa)\hbox{TOL}\\
\frac{V_0}{N_0}+\sum^L_{\ell=0} V_{\ell}N_\ell^{-1} \leq \left(\kappa \hbox{TOL}/C_\alpha\right)^2
\end{array}
\right.
\end{eqnarray*}
It is challenging to find a closed-form solution to the above optimization problem, and thus we will determine the level $L$ and $M_L$ directly from the bias constraint as follows: Let us split bias constraint into two bias constraints:
\begin{eqnarray*}
C_{2}h_L^{\eta_w} & \le & \frac{1}{2}(1-\kappa)\hbox{TOL}, \ \text{and} \\
C_{1}M_{L}^{-1} & \le & \frac{1}{2}(1-\kappa)\hbox{TOL}.
\end{eqnarray*}
By fixing $h_{\ell}=h_0\beta^{-\ell}$ as in \eqref{eq:hell}, we obtain the following values, denoted by $\kappa^*$ and $L^*$, for the balancing parameter $\kappa$ and the highest level $L$, respectively,
\begin{equation}
\label{eq:Lstar}
L^* = \left\lceil \eta_w^{-1}\left(\log_{\beta}\left( 2C_2h_0^{\eta_w}\right) + \log_{\beta}\left( \hbox{TOL}^{-1} \right) \right)\right\rceil,
\end{equation}
and 
\begin{equation}
\label{eq:kappastar}
\kappa^* = 1-C_2h_{L^*}^{\eta_w}\hbox{TOL}^{-1}.
\end{equation}
The choice of $M_{L}$, denoted by $M^*_{L}$, is given as
\begin{equation}
\label{eq:MLstar}
M^*_L = \left\lceil \frac{C_1}{1-\kappa^*}\hbox{TOL}^{-1} \right\rceil.
\end{equation}
Next, for $\ell = 0,\dots,L-1$ let the choice of $M_{\ell}$, denoted by $M^*_{\ell}$, be given by
\begin{equation}
\label{eq:Massump}
M^*_{\ell}=M^*_{L},
\end{equation}
which results in simplifying cancellations in the variance estimate \eqref{eq:variance.estimate}, leading to, for $\ell>0$,
\begin{equation}
V_{\ell} \lessapprox \mathcal{C}\dfrac{h^{2\eta_s}_\ell}{M_{L}} + h_\ell^{2\eta_w},
\end{equation}
for some $\mathcal{C}>0$. Given $\kappa^*$, $M^*_\ell$, and $L^*$, the optimal number of outer samples $N_{\ell}$ is the solution of
\begin{eqnarray*}
 N_{\ell}^*=\argmin_{N_{\ell}} && \sum^{L^*}_{\ell=0} N_{\ell}M^*_{\ell}h_{\ell}^{-\gamma} \\
\hbox{subject to}\;&& \frac{V_0}{N_0}+\sum^{L^*}_{\ell=0} V_{\ell}N_\ell^{-1} \leq \left(\kappa^* \hbox{TOL}/C_\alpha\right)^2
\end{eqnarray*}
which is given by
\begin{equation}
\label{eq:Nel}
N_{\ell}^*=\left\lceil \left( \frac{C_\alpha}{\kappa^*\hbox{TOL}} \right)^2 \sqrt{\frac{V_\ell}{M^*_{\ell}W(\bm{g}_{\ell})}}\left(\sum^{L^*}_{\ell=0} \sqrt{V_{\ell}M^*_{\ell}W(\bm{g}_{\ell})}\right) \right\rceil,
\end{equation}
as shown for standard MLMC in \cite{G2015}; the ceiling of the optimal solution is to ensure $N_{\ell}^{*}$ is a positive integer.
In conclusion, the proposed MLDLMC is given by \eqref{eq:mldl} with number of outer samples $\{ N^*_{\ell} \}_{\ell=0}^{L}$ and number of inner samples $\{ M^*_{\ell} \}_{\ell=0}^{L}$ for level $L=L^*$ such that \eqref{eq:probframe} is satisfied as $\hbox{TOL} \to 0$, i.e., for some desired error tolerance $\hbox{TOL}>0$, the random estimator is designed to satisfy
\begin{equation*}
\mathbb{P}\left(\vert I-\mathcal{I} \vert \leq \hbox{TOL} \right) \geq 1-\alpha,
\end{equation*}
with probability $1-\alpha$, where $0<\alpha \ll 1$.
The average work of the proposed MLDLMC can be bounded from above as follows:
\begin{equation}
\label{eq:workmldlmc}
W(\mathcal{I}_{\text{MLDLMC}}) \lessapprox \left( \frac{C_\alpha}{\kappa^*\hbox{TOL}} \right)^2 \left(\sum^{L^*}_{\ell=0} \sqrt{V_{\ell}M^*_{\ell}W(\bm{g}_{\ell})}\right)^2+\sum^{L^*}_{\ell=0} M^*_{\ell}W(\bm{g}_{\ell}).
\end{equation}
Also, an alternative to \eqref{eq:Massump}, namely to use $M^*_{\ell}=M_L$ for all $\ell$, is to numerically determine $\{ M_{\ell}^* \}^{L^*-1}_{\ell=0}$ by minimizing the upper bound of the work, \eqref{eq:workmldlmc}, with respect to $\{ M_{\ell}^* \}^{L^*-1}_{\ell=0}$.

\subsection{Computational work discussion}
\label{sec:compdiscuss}
There is a connection between the average work of the proposed MLDLMC (given in Section \ref{sec:mldlmcparam}) and that of the standard MLMC. The upper bound of the work of the MLDLMC with $L+1$ levels can be bounded from above as follows:
\begin{eqnarray}
W(\mathcal{I}_{\text{MLDLMC}}) &\propto& \sum^{L}_{\ell=0} N^*_{\ell}M^*_{\ell}h_{\ell}^{-\gamma} \nonumber \\
&\le& \left( \frac{C_\alpha}{\kappa\hbox{TOL}} \right)^2 \left(\sum^{L}_{\ell=0} \sqrt{V_{\ell}M^*_{L}W(\bm{g}_{\ell})}\right)^2+M^*_{L}\sum^{L}_{\ell=0}W(\bm{g}_{\ell}) \nonumber \\
&\lessapprox& \left( \frac{C_\alpha}{\kappa\hbox{TOL}} \right)^2 \left(\sum^{L}_{\ell=0} \sqrt{V_{\ell}M^*_{L}W(\bm{g}_{\ell})}\right)^2 \nonumber \\
&\lessapprox& \left( \frac{C_\alpha}{\kappa\hbox{TOL}} \right)^2 \left(\sqrt{M^*_LW(\bm{g}_0)}+\sum^{L}_{\ell=1} \sqrt{(h^{2\eta_s}_\ell +\mathcal{C}M^*_{L}h_\ell^{2\eta_w})W(\bm{g}_{\ell})}\right)^2 \nonumber \\
&\lessapprox& \left( \frac{C_\alpha}{\kappa\hbox{TOL}} \right)^2 \left(\sum^{L}_{\ell=1} \sqrt{h^{2\eta_s-\gamma}_\ell + \mathcal{C}\hbox{TOL}^{-1}h_\ell^{2\eta_w-\gamma}}\right)^2.\label{eq:optwork}
\end{eqnarray}
The constant $\mathcal{C}>0$ typically satisfies $\mathcal{C} \ll 1$ due to the Laplace-based importance sampling, and from the upper bound of the work we can notice that total work of the proposed MLDLMC will behave as standard MLMC as long as $\hbox{TOL}$ is large enough such that $\mathcal{C}\hbox{TOL}^{-1} \ll 1$. The asymptotic work complexity of the standard MLMC with respect to $\hbox{TOL}$ can be found in Theorem 2.1 \cite{G2015}. Note that asymptotically as $\hbox{TOL} \to 0$, we see from the work bound \eqref{eq:optwork} that MLDLMC exhibits an asymptotically worse complexity compared to that of MLMC, because of the additional term $\mathcal{C}\hbox{TOL}^{-1}h_\ell^{2\eta_w-\gamma}$ in \eqref{eq:optwork}, as expected due to the nested expectation of the EIG criterion. For the numerical example in Section \ref{sec:eitnum}, we observe that the computational work of the proposed MLDLMC \eqref{eq:mldl} follows $\hbox{TOL}^{-2}$ up to some logarithmic factor over a reasonable range of $\hbox{TOL}$, which is the optimal work rate of the standard MLMC under Assumption \ref{assump:g}. For cases with $2\eta_s>\gamma$, we thus expect that the work of MLDLMC follows the optimal rate of $\hbox{TOL}^{-2}$, for ranges of $\hbox{TOL}$ satisfying $\mathcal{C}\hbox{TOL}^{-1} \ll 1$, which could be satisfied thanks to the efficiency of the Laplace-based importance sampling.

\section{Multilevel Double Loop Stochastic Collocation (MLDLSC}
\label{sec:mldlscsec}
As an alternative to methods based on MC sampling, we propose a Multilevel Double Loop Stochastic Collocation (MLDLSC) method, based on the Multi-Index Stochastic Collocation (MISC) algorithm \cite{ali2016b,ali2016}, which exploits the regularity of the dependence on the random input variables. The idea is to compute the telescopic sum differences, i.e., expectations, in the multilevel estimator \eqref{eq:telescop}, by stochastic collocation, which is a high-dimensional integration over the probability space achieved by deterministic quadrature on sparse grids, see, e.g., \cite{BNT2010,BNTT2011}.

\subsection{Multilevel stochastic collocation}
\label{sec:mlsc}
We start by defining a quadrature operator for a one-dimensional real-valued continuous function $u: \Gamma_i \to \mathbb{R}$, where $\Gamma_i=[-1,1]$ is any of the univariate sub-domains $\Gamma_1,\ldots,\Gamma_d$ composing the complete $d$-dimensional domain $\bm{\Gamma} \myeq \Gamma_1 \times \cdots \times \Gamma_d$. The quadrature operator is defined as
\begin{equation}
\mathcal{Q}^{m(\beta)}: C^0(\Gamma_i) \to \mathbb{R}, \quad \mathcal{Q}^{m(\beta)}[u]=\sum^{m(\beta)}_{j=1} u(z_{\beta,j})\omega_{\beta,j},
\end{equation}
where $\beta$ is a positive integer specifying the ``level'' of the quadrature operator, $m(\beta)$ a strictly increasing function giving the number of distinct collocation points, $\{ z_{\beta,j} \}_{j=1}^{m(\beta)}$, and $z_{\beta,j} \in \Gamma_i$ with corresponding weights $\{ \omega_{\beta,j} \}^{m(\beta)}_{j=1}$. The collocation points are chosen according to the underlying probability distribution; see \cite{XK2002}. For the uniform probability distribution, we adopt the Clenshaw-Curtis family of points and weights, which has the desired property of being \emph{nested}. The distribution of points is given by
\begin{equation*}
z_{\beta,j}=\cos\left( \dfrac{(j-1)\pi}{m(\beta)-1} \right), \quad 1 \le j \le m(\beta),
\end{equation*}
where the function $m(\beta)$ is defined as $m(\beta)=2^{\beta-1}+1$ for $\beta \ge 2$, where $m(0)=0$, $m(1)=1$.

The generalization to high-dimensional real-valued continuous functions $u: \bm{\Gamma} \to \mathbb{R}$ is obtained by introducing a quadrature operator that is a tensorization of the one-dimensional quadrature operators, i.e.,
\begin{equation*}
\mathcal{Q}^{\bm{m}(\bm{\beta})}: C^{0}(\bm{\Gamma}) \to \mathbb{R}, \quad \mathcal{Q}^{\bm{m}(\bm{\beta})} = \bigotimes_{1 \le i \le d} \mathcal{Q}^{m_i(\beta_i)}, \quad \mathcal{Q}^{\bm{m}(\bm{\beta})}[u]=\sum^{\# \bm{m}(\bm{\beta})}_{j=1} u(\bm{z}_j)\omega_j,
\end{equation*}
where $\bm{z}_j$ are the points on the tensor grid $\bigotimes_{1 \le i \le d} \{z_{\beta_i,j}\}^{m_i(\beta_i)}_{j=1}$, $\omega_j$ are the products of the weights imposed by the one-dimensional quadrature rules, $m_i(\bm{\beta})$ is the function giving the number of collocation points for input direction $i$, and $\# \bm{m}(\bm{\beta})$ denotes the total number of collocation points on the full grid for a multi-index $\bm{\beta}$, i.e., $\# \bm{m}(\bm{\beta}) \myeq \prod^{d}_{i=1} m_i(\beta_i)$. A hierarchy of the anisotropic full-tensor approximations can be constructed by selecting $\bm{\beta} \in \mathbb{N}^d$ such that
\begin{equation*}
\lceil s_i\beta_i \rceil=w,
\end{equation*}
for the sequence of approximation levels $w \in \mathbb{N}$, and where $s_i$ is a user-specified importance weight for input direction $i$. This is known as the \emph{total product} (TP) approximation. However, this leads to the total number of collocation points growing exponentially as $w$ increases. To mitigate the \emph{curse of dimensionality}, we adopt a \emph{sparsification} technique, known as sparse grid stochastic collocation (SGSC), see, e.g., \cite{BNT2010,BNTT2011}.

The TP approximation for integration is denoted by $U_{\bm{\beta}} \myeq \mathcal{Q}^{\bm{m}(\bm{\beta})}[u]$. The SC quadrature uses the difference operator, $\Delta_i$, and is given for $1 \le i \le d$ by
\begin{align}
  \label{eq:deltai}
   \Delta_i[U_{\bm{\beta}}] & \myeq
    \begin{cases}
      U_{\bm{\beta}}-U_{\bm{\beta}-\bm{e}_i}, & \text{if $\beta_i>1$}\\
      U_{\bm{\beta}}, & \text{if $\beta_i=1$},
    \end{cases}
\end{align}
where $(\bm{e}_i)_{k}=1$ if $i=k$, and zero otherwise. The sparse-grid stochastic collocation quadrature can be formulated as
\begin{equation}
\mathcal{I}_{\text{SGSC}}=\sum_{\bm{\beta} \in \Lambda} \bm{\Delta}[U_{\bm{\beta}}]=\sum_{\bm{\beta} \in \Lambda} \sum_{\substack{\bm{j} \in \{0,1\}^{d} \\ \bm{\beta}+\bm{j} \in \Lambda}} (-1)^{\vert \bm{j} \vert} U_{\bm{\beta}},
\end{equation}
for some multi-index set $\Lambda \subset \mathbb{N}^{d}$, and the mixed-difference operator is given by
\begin{eqnarray*}
\bm{\Delta}[U_{\bm{\beta}}] & \myeq & \bigotimes_{1 \le i \le d} \Delta_i[U_{\bm{\beta}}] \myeq \Delta_1 \left[ \Delta_2 \left[ \ldots \Delta_d[U_{\bm{\beta}}] \right] \right] \\
\quad & = & \sum_{\bm{j} \in \{0,1\}^{d}} (-1)^{\vert \bm{j} \vert}U_{\bm{\beta}-\bm{j}}.
\end{eqnarray*}

Now, we consider a case in which $u$ is numerically approximated by $u_{\ell}$ at a discretization level $\ell$ defined by the mesh-element size $h_{\ell}$. Therefore, the complete sparse hierarchy can thus be specified by $\ell$ in the physical space and by $\bm{\beta}$ in the probability space, which leads us to the MLSC estimator of $\mathbb{E}[u]$, given by
\begin{equation}
\label{eq:mlsc}
\mathcal{I}_{\text{MLSC}} \myeq \sum_{[\ell,\bm{\beta}] \in \Lambda} \bm{\Delta}^{\text{mix}}[U_{\ell,\bm{\beta}}]=\sum_{[\ell,\bm{\beta}] \in \Lambda} \sum_{\substack{\bm{j} \in \{0,1\}^{d+1}\\ [\ell,\bm{\beta}]+\bm{j} \in \Lambda}} (-1)^{\vert \bm{j} \vert} U_{\ell,\bm{\beta}},
\end{equation}
where $\Lambda \subset \mathbb{N}^{d+1}$, $U_{\ell,\bm{\beta}} \myeq \mathcal{Q}^{\bm{m}(\bm{\beta})}[u_{\ell}]$, and
\begin{align}
  \label{eq:mixdelta}
   \Delta^{\text{mix}}[U_{\ell,\bm{\beta}}] & \myeq
    \begin{cases}
      \bm{\Delta}[U_{\ell,\bm{\beta}}-U_{\ell-1,\bm{\beta}}], & \text{if $\ell>0$}\\
      \bm{\Delta}[U_{\ell,\bm{\beta}}], & \text{if $\ell=0$}.
    \end{cases}
\end{align}
We evaluate MLSC by computing the full-tensor approximations $U_{\ell,\bm{\beta}}$ independently, and combining them linearly according to the combination technique \eqref{eq:mlsc}. Of course, the effectiveness depends on the choice of the multi-index set $\Lambda$. The idea behind the \emph{sparse} construction is that $\Lambda$ should be chosen to exclude ``expensive'' isotropic full-tensor
approximations from the estimate, by refining only a subset of the physical or probability directions simultaneously. Then, we combine these approximations using the combination-technique formula \eqref{eq:mlsc} to create a more accurate approximation. Various approaches have been proposed for selecting the multi-index set $\Lambda$, such as using the classical sets given in \cite{BNT2010} or selecting the set adaptively as discussed in \cite{BTT2018,GG2003,NTTT2016,SS2013}.

\subsection{Multilevel Double Loop Stochastic Collocation (MLDLSC) estimator}
\label{sec:miscis}

Here, we recast EIG \eqref{expecinfgain} into an integration with respect to $\Epsilon$ instead of $\bm{Y}$:
\begin{eqnarray}
\label{eq:eigepsilon}
I&=&\int_{\Theta} \int_{\mathcal{E}} \log{\left(\dfrac{\pi_\epsilon(\Epsilon)}{\int_{\Theta} p(\bm{G}(\bm{\theta})+\Epsilon \vert \tilde{\bm{\theta}})\pi(\tilde{\bm{\theta}})d\tilde{\bm{\theta}}}\right)} \pi_\epsilon(\Epsilon) \pi(\bm{\theta}) \,d\Epsilon \,d\bm{\theta} \\
&=& \int_{\Theta} \int_{\mathcal{E}} \log{\left(\dfrac{\pi_\epsilon(\Epsilon)}{\int_{\Theta} p(\bm{G}(\bm{\theta})+\Epsilon \vert \tilde{\bm{\theta}})R(\tilde{\bm{\theta}};\bm{G}(\bm{\theta})+\Epsilon)\tilde{\pi}(\tilde{\bm{\theta}} \vert \bm{G}(\bm{\theta})+\Epsilon) d\tilde{\bm{\theta}}}\right)} \pi_\epsilon(\Epsilon) \pi(\bm{\theta}) \,d\Epsilon \,d\bm{\theta},
\end{eqnarray}
where $\pi_\epsilon(\Epsilon)=\prod^{d}_{i=1}\pi_\epsilon(\bm{\epsilon}_i)$, and the likelihood ratio is $R(\tilde{\bm{\theta}};\bm{G}(\bm{\theta})+\Epsilon)=\pi(\tilde{\bm{\theta}})/\tilde{\pi}(\tilde{\bm{\theta}} \vert \bm{G}(\bm{\theta})+\Epsilon)$ as defined in \eqref{eq:Rell}, the importance sampling distribution is $\tilde{\pi}(\tilde{\bm{\theta}} \vert \bm{G}(\bm{\theta})+\Epsilon) \sim \mathcal{N}(\hat{\bm{\theta}}(\bm{G}(\bm{\theta})+\Epsilon),\bm{\Sigma}(\bm{G}(\bm{\theta})+\Epsilon))$ as defined in \eqref{eq:istilde} with the MAP estimate, $\hat{\bm{\theta}}(\bm{G}(\bm{\theta})+\Epsilon)$, as given in \eqref{fittheta}, and the approximate covariance is $\bm{\Sigma}(\bm{G}(\bm{\theta})+\Epsilon)$ as given in \eqref{eq:sigmahat}. We introduce the auxiliary function,
\begin{equation*}
\Psi_{\ell}(\tilde{\bm{\theta}};\bm{G}_{\ell}(\bm{\theta})+\Epsilon) \myeq p_{\ell}(\bm{G}(\bm{\theta})+\Epsilon \vert \tilde{\bm{\theta}})R_{\ell}(\tilde{\bm{\theta}};\bm{G}_{\ell}(\bm{\theta})+\Epsilon).
\end{equation*}
Furthermore, we let $\bm{\beta}=\left(\bm{\beta}_1,\bm{\beta}_2\right)$ and $\bm{m}_{\bm{\beta}}=\left(\bm{m}_{\bm{\beta}_1},\bm{m}_{\bm{\beta}_2}\right)$, where $\bm{\beta}_1$,$\bm{\beta}_2$ are multi-indices associated with the random variables of the outer $(\bm{\theta},\bm{\epsilon}_i)$ and inner ($\bm{\tilde{\theta}}$) integrals, respectively. The proposed MLDLSC estimator for approximating the EIG in the form given in \eqref{eq:eigepsilon} is
\begin{equation}
\label{eq:mldlsc}
\mathcal{I}_{\text{MLDLSC}} \myeq \sum_{[l,\bm{\beta}] \in \Lambda} \bm{\Delta}\left[ F_{\ell,\bm{\beta}} \right]=\sum_{[\ell,\bm{\beta}] \in \Lambda} \sum_{\substack{\bm{j} \in \{0,1\}^{d+1}\\ [\ell,\bm{\beta}]+\bm{j} \in \Lambda}} (-1)^{\vert \bm{j} \vert} F_{\ell,\bm{\beta}},
\end{equation}
where $F_{\ell,\bm{\beta}} \myeq \mathcal{Q}^{\bm{m}_{\bm{\beta}_1}}\left[ \tilde{f}_{\ell,\bm{\beta}_2} \right]$, $\Lambda \subset \mathbb{N}^{d+1}$ is the multi-index set, and
\begin{equation}
\tilde{f}_{\ell,\bm{\beta}_2}(\Epsilon,\bm{\theta}) \myeq \log{\left( \dfrac{\pi_\epsilon(\Epsilon)}{\mathcal{Q}^{\bm{m}_{\bm{\beta}_2}}[\Psi_{\ell}]} \right)}.
\end{equation}

The natural choice of collocation points and weights are Gauss-Hermite for a random variable $\bm{\epsilon}_i \sim \mathcal{N}(\bm{0},\bm{\Sigma}_{\bm{\epsilon}_i})$. If the covariance, $\bm{\Sigma}$, is a positive-definite non-diagonal matrix, then we can transform the standard $\mathcal{N}(\bm{0},\bm{1})$ Gauss-Hermite points, here denoted by $\bm{z}_{\mathcal{N}(\bm{0},\bm{1})}$, to $\mathcal{N}(\bm{\mu},\bm{\Sigma})$ Gauss-Hermite points, denoted by $\bm{z}$, by following two steps:
\begin{eqnarray}
\bm{\Sigma} &=& \bm{L}\bm{L}^{T} \\
\bm{z} &=& \bm{L}\bm{z}_{\mathcal{N}(0,1)}+\bm{\mu},
\end{eqnarray}
where $\bm{L}$ is a left triangular matrix from the Cholesky decomposition. Similarly, the Gauss-Hermite points can be calculated from the standard Gauss-Hermite points for the random variables $\tilde{\bm{\theta}}$ follow the the importance sampling PDF, $\tilde{\pi}_{\ell}$. An efficient importance sampling measure for the inner expectation of the stochastic collocation approach is necessary to avoid numerical underflow and, as demonstrated in \cite{BDELT2018}, the Laplace-based importance sampling is adequate. The collocation points and weights for $\bm{\theta}$ are chosen with respect to $\pi(\bm{\theta})$.

\section{Optimal electrodes placement in electrical impedance tomography}
\label{sec:eitnum}
We apply the methods, MLDLMC and MLDLSC, proposed in this work, and the optimized DLMC method with Laplace-based importance sampling (DLMCIS) in \cite{BDELT2018}, for estimating the expected information gain for an electrical impedance tomography (EIT) experiment. EIT is an technique for imaging the interior conductivity of a closed body based on the voltage measurements from electrodes placed on the body's free-surface. In this experiment, low-frequency electrical currents are injected through electrodes attached to a composite laminate material made of four orthotropic plies. The potential field in the body of the material is considered quasi-static for a given conductivity.

\subsection{Configuration of the experiment}
The complete electrode model (CEM) \cite{somersalo} is used to formulate the problem. The composite body $D$, with boundary $\partial D$, is formed of $N_p$ plies, i.e. $D = \displaystyle{\cup^{N_p}_{k=1} D_k}$. The configuration is such that the plies overlap with their fibers facing different directions. The upper and lower surfaces of the boundary $\partial D$ are equipped with $N_{_{\tiny{\hbox{el}}}}$ square-shaped electrodes $E_l$, $l=1,\cdots,N_{_{\tiny{\hbox{el}}}}$, with dimensions $e_{el}$ and a surface impedance of $z_l$, as illustrated in Figure \ref{fg:setup}.
\begin{figure}[H]
\centering
\includegraphics[width=0.55\textwidth]{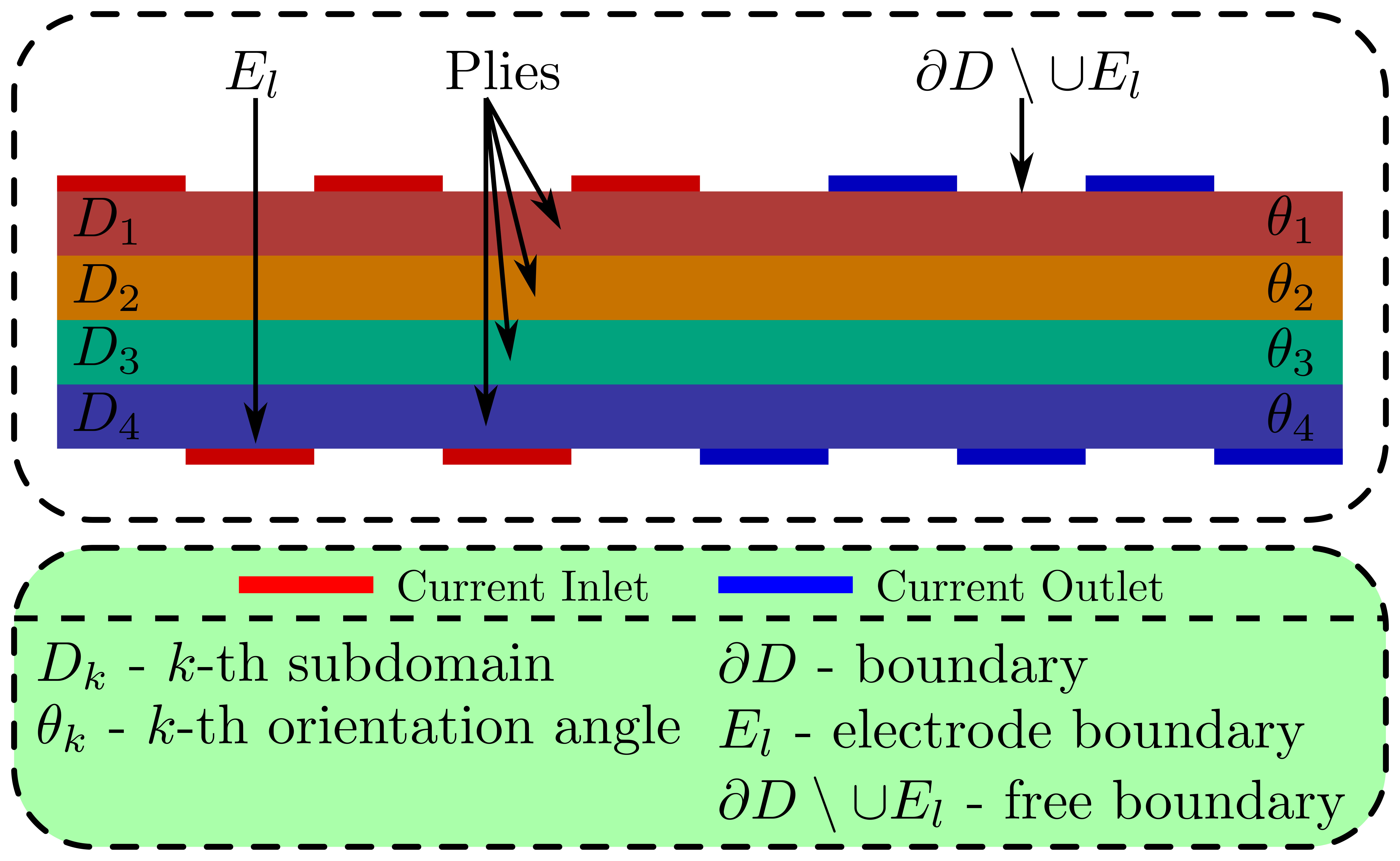}
\caption{(Color online) Experimental set-up (cross section) of a composite plate to be examined made up of four plies with the same thickness. Five electrodes are placed on the top surface and five on the bottom surface. Red (blue) electrodes represent the current inlet (outlet). The current at the electrodes are known while the potential field on the body $\mathcal{D}$ and the potential at the electrodes $E_l$ are unknowns.}
\label{fg:setup}
\end{figure}
The potential field $u$ obeys
\begin{eqnarray}
\label{eq_eit1}
\nabla \cdot \bm{\jmath}(\omega, \bm{x}) \!\!\! & = & \!\!\! 0,\;\;\; \hbox{in}\;\;\; \mathcal{D}, \text{and} \\
\bm{\jmath}(\omega, \bm{x}) \!\!\! & = & \!\!\! \bm{\bar{\sigma}}(\omega,\bm{x}) \cdot \nabla u(\omega,\bm{x}),
\label{eq_eit2}
\end{eqnarray}
where $\bm{\jmath}$ is the flux of electric current, and $\bm{\bar{\sigma}}$ is the conductivity field and is given by
\begin{eqnarray*}
\bm{\bar{\sigma}}(\omega,\bm{x}) = \bm{Q}^T(\theta_k(\omega)) \cdot \bm{\sigma} \cdot \bm{Q}(\theta_k(\omega)),\;\;\;\hbox{for}\;\;\; \bm{x} \in D_k,\;\; k=1,\cdots,N_{_p}.
\end{eqnarray*}
The CEM is a set of boundary conditions for \eqref{eq_eit1}-\eqref{eq_eit2} given by
\begin{eqnarray}
\label{eq_cem1}
\left\{
\begin{array}{lll}
\displaystyle{\bm{\jmath} \cdot \bm{n} = 0,\;\;\; \text{on}\;\;\; \partial D \backslash \left( \cup E_l\right),}\\
\displaystyle{\int _{E_l} \bm{\jmath} \cdot \bm{n} \, d \bm{x} = I_l \;\;\;  \text{on} \;\;\;  E_l, \;\;\; l= 1,\cdots,N_{_{\tiny{\hbox{el}}}},}\\
\displaystyle{\frac{1}{E_l}\int_{E_l} u \,d\bm{x} + z_l \int_{E_l}\bm{\jmath} \cdot \bm{n} \,d\bm{x}  = U_l \;\;\;\hbox{on} \;\;\;  E_l , \;\;\; l= 1,\cdots,N_{_{\tiny{\hbox{el}}}}},
\end{array}
\right.
\end{eqnarray}
where $\bm{n}$ represents the outward normal unit vector. To obtain well-posedness (existence and uniqueness of $(u, \bm{U})$), the Kirchhoff law of charge conservation and the ground potential condition,
\begin{eqnarray}
\label{eq_constr}
\displaystyle{ \sum_{l=1}^{N_{_{\tiny{\hbox{el}}}}} I_l = 0} \;\;\; \hbox{and}\;\;\;
\displaystyle{ \sum_{l=1}^{N_{_{\tiny{\hbox{el}}}}} U_l = 0,}
\end{eqnarray}
respectively, are constraints.
The orthogonal matrix $\bm{Q}(\theta_k)$ is a rotational matrix that defines the orientation of the fibers, in ply $k$ at a given angle $\theta_k$, while $\bm{\sigma}$ stands for the orthotropic conductivity, i.e.,
\begin{eqnarray*}
\bm{Q}(\theta_k) =\begin{bmatrix}
\cos(\theta_k)  &     0      &      -\sin(\theta_k) \\
0                       &     1     &        0 \\
\sin(\theta_k)   &     0     &        \cos(\theta_k)
\end{bmatrix}
\;\;\;\hbox{and}\;\;\;
\bm{\sigma} =\begin{bmatrix}
\sigma_{1}      &       0               &       0 \\
0                     &  \sigma_{2}     & 0 \\
0                   &         0               & \sigma_{3}
\end{bmatrix}.
\end{eqnarray*}

In the rest of the paper, the EIT model refers to \eqref{eq_eit1}, \eqref{eq_eit2}, \eqref{eq_cem1}, and \eqref{eq_constr}. The conductivity  $\bm{\bar{\sigma}}$ is random and assumed to be a uniformly and strictly positive element of $L^\infty(\Omega\times D)$ in order to guarantee ellipticity. The vectors $\bm{I} = \left( I_1,I_2,\cdots,I_{N_{_{\tiny{\hbox{el}}}}} \right)$, and $\bm{U} = \left( U_1,U_2,\cdots,U_{N_{_{\tiny{\hbox{el}}}}} \right)$ respectively determine the vector of the injected (deterministic) current and the vector measurement of the (random) potential at the electrodes. According to the constraints in \eqref{eq_constr}, $\bm{I}$ belongs to the mean-free subspace $\mathbb{R}^{N_{_{\tiny{\hbox{el}}}}}_{_{\tiny{\hbox{free}}}}$ of $\mathbb{R}^{N_{_{\tiny{\hbox{el}}}}}$ and $\bm{U}$ is an element of $\mathbb{R}^{N_{_{\tiny{\hbox{el}}}}}_{_{\tiny{\hbox{free}}}}$. In solving \eqref{eq_eit1} with the consitutive relation \eqref{eq_eit2} subject to the second condition of \eqref{eq_cem1} (i.e., the assigned current at the electrodes) and \eqref{eq_constr}. The unknowns are represented by the pair of potential field on $\mathcal{D}$ and potential at the electrodes $E_l$, that is, $\left(u,\bm{U}\right)$.

\subsection{Experimental design formulation}
\label{sec:designform}
Ten electrodes are placed on the composite laminate body of four ($N_p=4$) plies, with five placed on the top ply and five on the bottom ply, to measure the electrical potential $U_l$ at the electrodes. The parameters of the four plies are $\sigma_{11}=0.05$, $\sigma_{22}=\sigma_{33}=10^{-3}$, and $z_l=0.1$. In Figure \ref{fg:setup}, the red-filled rectangles on the plies represent the electrodes where current is injected and the blue ones represent the electrodes where the current exits. The inlet and outlet currents are in absolute value equal to $0.1$. The orientations of the angles $\theta_1$, $\theta_2$, $\theta_3$, and $\theta_4$ of the fibers are the uncertain parameters targeted for the statistical inference, i.e., the experimental goal, and we consider the following uniform distributions to describe our prior knowledge:
\begin{align*}
\pi(\theta_1)\sim\mathcal{U}\left(\frac{\pi}{3}-0.05,\frac{\pi}{3}+0.05\right), \quad &
\pi(\theta_2)\sim\mathcal{U}\left(\frac{\pi}{4}-0.05,\frac{\pi}{4}+0.05\right),&\\
 \pi(\theta_3)\sim\mathcal{U}\left(\frac{\pi}{5}-0.05,\frac{\pi}{5}+0.05\right),\quad &
 \pi(\theta_4)\sim\mathcal{U}\left(\frac{\pi}{6}-0.05,\frac{\pi}{6}+0.05\right).&
\end{align*}

For the Bayesian experimental design problem, the data model (cf. \eqref{eq_datamodel}) is given by
\begin{eqnarray}
\bm{y}_i = \bm{U}_h(\bm{\theta}) +\bm{\epsilon}_i, \quad \text{for} \quad i = 1, \cdots, N_e \,,
\end{eqnarray}
where $\bm{y}_i \in \mathbb{R}^{N_{_{\tiny{\hbox{el}}}}-1}$, i.e., $q=N_{_{\tiny{\hbox{el}}}}-1$, and the error distribution is Gaussian, i.e., $\bm{\epsilon}_i \sim \mathcal{N}(0,10^{-4}\bm{1}_{q \times q})$. No repeated experiments are considered, i.e., $N_e=1$. The vector $\bm{\theta}=(\theta_{1},\theta_{2},\theta_{3},\theta_{4})$ represents the unknown orientation angles that we want to know, $\bm{U}_h = (U_1, \cdots,U_{N_{_{\tiny{\hbox{el}}}}-1})$ is a finite elements approximation, in the Galerkin sense, of $\bm{U}$ from the variational problem of finding $\left(u,\bm{U}\right) \in  L^2_{\mathbb{P}} \left(\Omega;\mathcal{H} \right)$ such that
\begin{eqnarray}\label{eq:pde}
\mathbb{E} \left[B\left((u,\bm{U}),(v,\bm{V})\right) \right] = \bm{I}_e \cdot \mathbb{E} \left[ \bm{U} \right], \;\;\;\;\hbox{for all}\;\;\;(v,\bm{V}) \in L^2_{\mathbb{P}} \left(\Omega;\mathcal{H} \right),
\end{eqnarray}
and where, for any event $\omega \in \Omega$, the bilinear form $B:\mathcal{H}\times\mathcal{H} \rightarrow \mathbb{R}$ is
\begin{eqnarray}
B\left((u,\bm{U}),(v,\bm{V})\right) = \int_D \bm{\jmath} \cdot \nabla v d D + \sum_{l=1}^{N_{_{\tiny{\hbox{el}}}}} \frac{1}{z_l} \int_{E_l} \left(U_l -u \right) \left(V_l - v \right) d E_l,
\end{eqnarray}
where $(\Omega,\mathcal{F},\mathbb{P})$ stands for the complete probability space, $\mathcal{F}$ is the $\sigma$-field of events, $\mathbb{P} : \mathcal{F} \rightarrow [0,1]$ is the probability measure, and $\Omega$ is the set of outcomes. The space of the solution for the potential field $(u(\omega),\bm{U}(\omega))$ is  $\mathcal{H} \myeq H^1(D) \times \mathbb{R}^{N_{_{\tiny{\hbox{el}}}}}_{_{\tiny{\hbox{free}}}}$ for a given random event $\omega \in \Omega$, and $L^2_{\mathbb{P}} \left(\Omega;\mathcal{H} \right)$ is  the Bochner space given by
\begin{eqnarray*}
L^2_{\mathbb{P}} \left(\Omega;\mathcal{H} \right) \myeq \left\{ (u,\bm{U}): \Omega \rightarrow \mathcal{H} \;\;\;\hbox{s.t.}\;\; \int_{\Omega} \left\|(u(\omega),\bm{U}(\omega))\right\|^2_{\mathcal{H}} d \mathbb{P}(\omega) < \infty \right\}.
\end{eqnarray*}
Figure \ref{fg:FEM} (top) shows both the electric potential field and the current streamlines, while Figure \ref{fg:FEM} (bottom) shows the current streamlines along the four plies of the composite material. Streamlines connect electrodes with prescribed inlet flux and electrodes with prescribed outlet flux. Due to jump discontinuities in the conductivity parameter, the current flux lacks smoothness. A similar behavior occurs at the edges of the electrodes, where there is a sudden transition between the boundary conditions, no flux and current inlet/outlet flux. As for the potential field, since we account for a surface impedance $z_l$ to emulate imperfections when an electrode is attached to the surface of the plate, the potential at the electrodes slightly differs from the potential field at those boundaries.
\begin{figure}[H]
\centering
\includegraphics[width=0.65\textwidth]{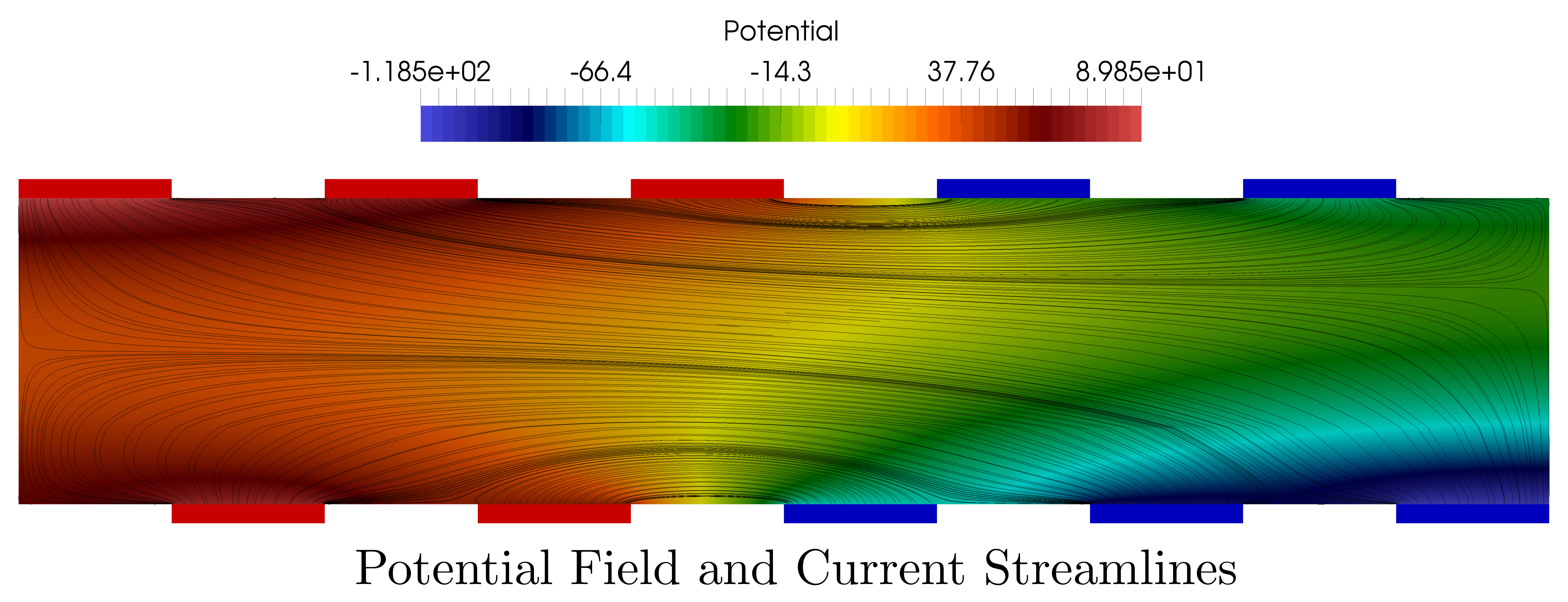}
\includegraphics[width=0.65\textwidth]{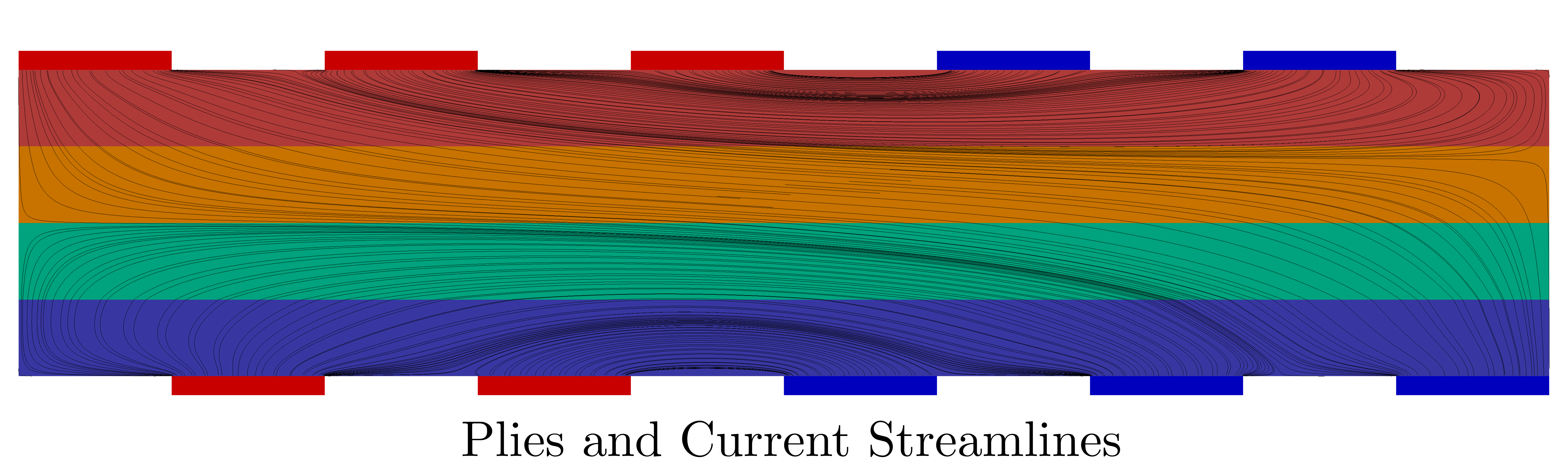}
\caption{(Color online) Finite elements approximation of the potential field (top) and streamlines (bottom). Streamlines depicting the current flux from inlet (red electrodes) to outlet (blue electrodes).}
\label{fg:FEM}
\end{figure}

\subsection{Implementation details of multilevel methods}
By least-squares estimation, we estimate for $\bm{g}_\ell$ the rate of work $\gamma \approx 2$ and the weak error rate $\eta_w \approx 1.5$ in Assumption \ref{assump:g}. The weak convergence of $\bm{g}_{\ell}$ is used as an estimate for the weak convergence of the approximate EIG, $I_{\ell}$ in \eqref{eq:klfl}. We define the sequence for the mesh hierarchy as $h_{\ell}=2^{-\ell}h_0$, $\ell>0$ where $h_0$ is the coarsest mesh-element size. In our numerical experiment, $h_0$ is the mesh size determined by a rectangular mesh ($N_x=10,N_y=4$) on a domain $L_x \times L_y$ ($L_x=20,L_y=4$).

The MLDLMC estimator is given by \eqref{eq:mldl}, and the method parameters are estimated according to the procedure given in Section \ref{sec:mldlmcparam}. The unknowns are the two constants $C_1$ and $C_2$ in the bias constraints and estimates are also required for $V_{\ell}$. Thus, we perform two runs of MLDLMC with $L=5$, $N_{\ell}=5$ for all levels $\ell$ with the identical samples for both runs, and then for one run $M_L=1$ and the other $M_L=10$. By using these initial runs and associated samples of $\bm{g}_{\ell}$ for $\ell=0,\ldots,L$, we can estimate $C_1,C_2,\eta_s$ and $V_{\ell}$. We note that $C_1 \approx 5 \times 10^{-3}$ for our example, and due to $C_1$ being smaller than any $\hbox{TOL}$ to be considered, we set $M_{L}=1$. By construction \eqref{eq:Massump}, we have that $M_{\ell}=1$ for all $\ell$. Given $C_2$, we can estimate $L$ and $\kappa$ from \eqref{eq:Lstar}. The variances $V_{\ell}$ are numerically estimated, and from from that the number of inner samples $N_{\ell}$ is determined from \eqref{eq:Nel}. Since $M_L$ is small, we can resort to a standard MLMC implementation as shown in Section \ref{sec:compdiscuss}. We will use the Continuation Multilevel Monte Carlo (CMLMC) given by Algorithm 2 \cite{key56} as it has the nice feature to adaptively estimate $V_{\ell}$ and accordingly determine $N_{\ell}$ until $\hbox{TOL}$ is satisfied.

The MLDLSC estimator is given by \eqref{eq:mldlsc} and the approach is described in detail in Section \ref{sec:miscis}. We adopt the MISC implementation given by Algorithm 1 of \cite{BTT2018}. The full-tensor approximation of the inner expectation with the change of measure being the Laplace-based importance sampling and we set the level to 0, which is analog to the choice of $M_L=1$ for MLDLMC. The random parameter $\bm{\theta}$ follows a uniform distribution in this case, and, thus, the Clenshaw--Curtis points and weights are used \cite{T2008}.

\subsection{Numerical results}
In this section, we analyze the performance of MLDLMC and MLDLSC for estimating the value of the EIG criterion for a fixed EIT experiment. For both estimators the goal is to achieve a specified error tolerance, $\hbox{TOL}$. As described in Section \ref{sec:mldlmcparam}, since MLDLMC is a random estimator, a relaxed probability constraint is applied, given in \eqref{eq:probframe} with $\alpha=0.05$, which should ensure that the absolute error is below $\hbox{TOL}$ with a $95\%$ probability of success. The performance in terms of computational efficiency and consistency will be investigated over a range of $\hbox{TOL}$.

In Figure \ref{fig:errvstol}, we numerically verify for $\hbox{TOL}$ ranging from $1$ to $0.001$ the consistency between $\hbox{TOL}$ and a computed absolute error using MLDLSC for $\hbox{TOL}^{-3}$. Even though in practice one would only perform a single or a few runs of MLDLMC, here we instead for analysis purposes repeat the random estimator MLDLMC hundred times for each considered $\hbox{TOL}$ using different pseudo-random states. We observe that only $2\%$ of the MLDLMC runs result in relative errors larger than $\hbox{TOL}$, which is consistent with our choice of $95\%$ probability of success. From the results shown in Figure \ref{errvstol.MLSC}, one observes that the estimation error of MLDLSC is consistently below $\hbox{TOL}$, but the error estimates underpinning the adaptivity algorithm of MLDLSC are not sharp.

\begin{figure}[H]
  \centering
  \subfloat[MLDLMC]{\label{errvstol.MLMC}\includegraphics[width=0.5\linewidth]{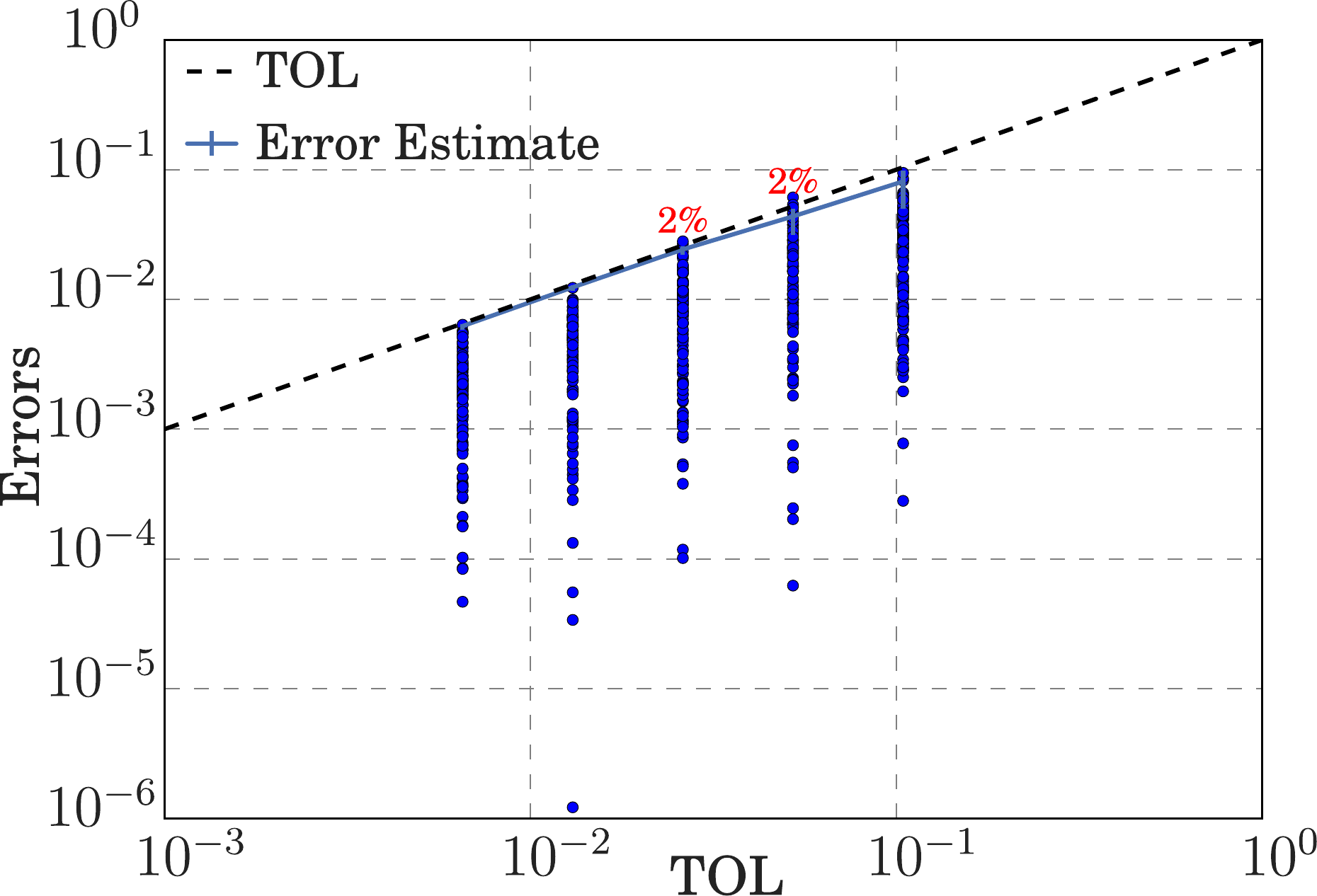}}
  \subfloat[MLDLSC]{\label{errvstol.MLSC}\includegraphics[width=0.5\linewidth]{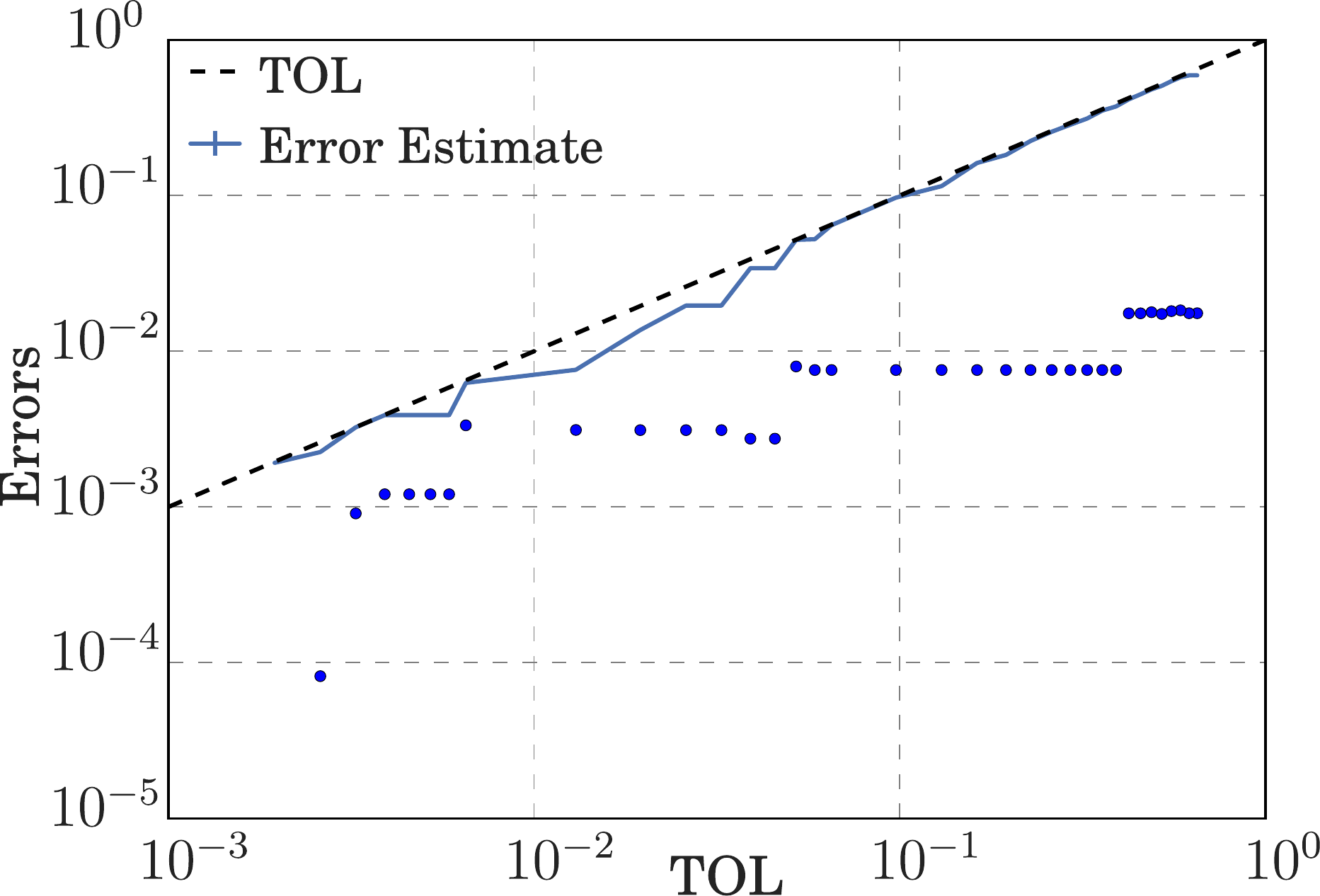}}
  \caption{Relative absolute error \emph{versus} error tolerance \hbox{TOL} for MLDLMC (left) and MLDLSC (right). The blue dots are individual runs. For each considered tolerance $\hbox{TOL}$, we perform 100 runs of the random estimator MLDLMC.}
  \label{fig:errvstol}
\end{figure}

We observe that the sample mean (Fig. \ref{fg:El_MLMC}) and sample variance (Fig. \ref{fg:Vl_MLMC}) of the telescopic differences in MLDLMC with respect to level $\ell$ decay at rates roughly equal to the assumed asymptotic rates $\eta_w=1$ and $2\eta_s=2$, respectively. These plots are used to show the consistency of the discretization method in terms of the weak and strong convergences stated in Assumption \ref{assump:g}.

\begin{figure}[H]
  \centering
  \subfloat[Weak convergence]{\label{fg:El_MLMC}\includegraphics[width=0.45\linewidth]{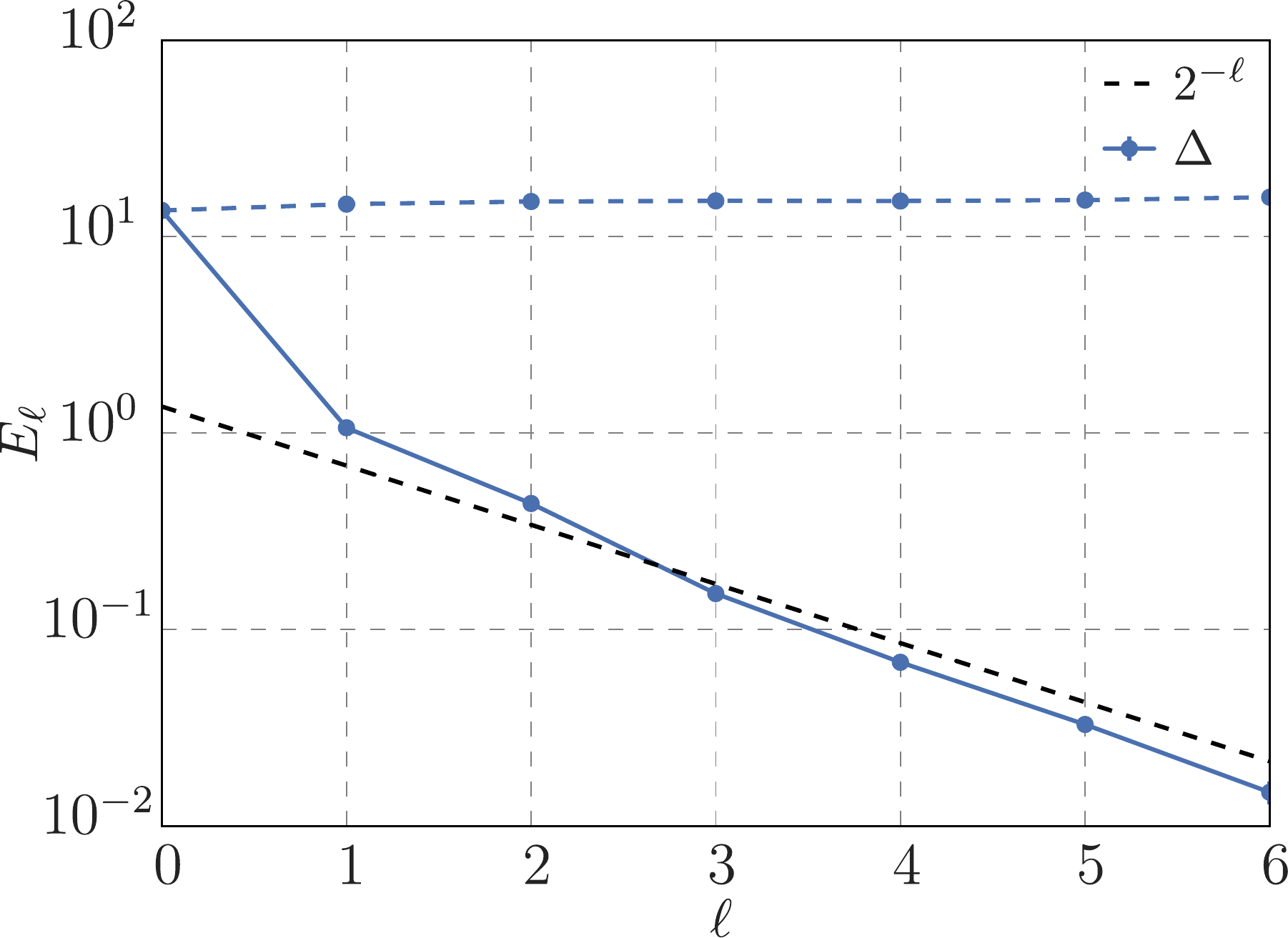}}
  \subfloat[Strong convergence]{\label{fg:Vl_MLMC}\includegraphics[width=0.45\linewidth]{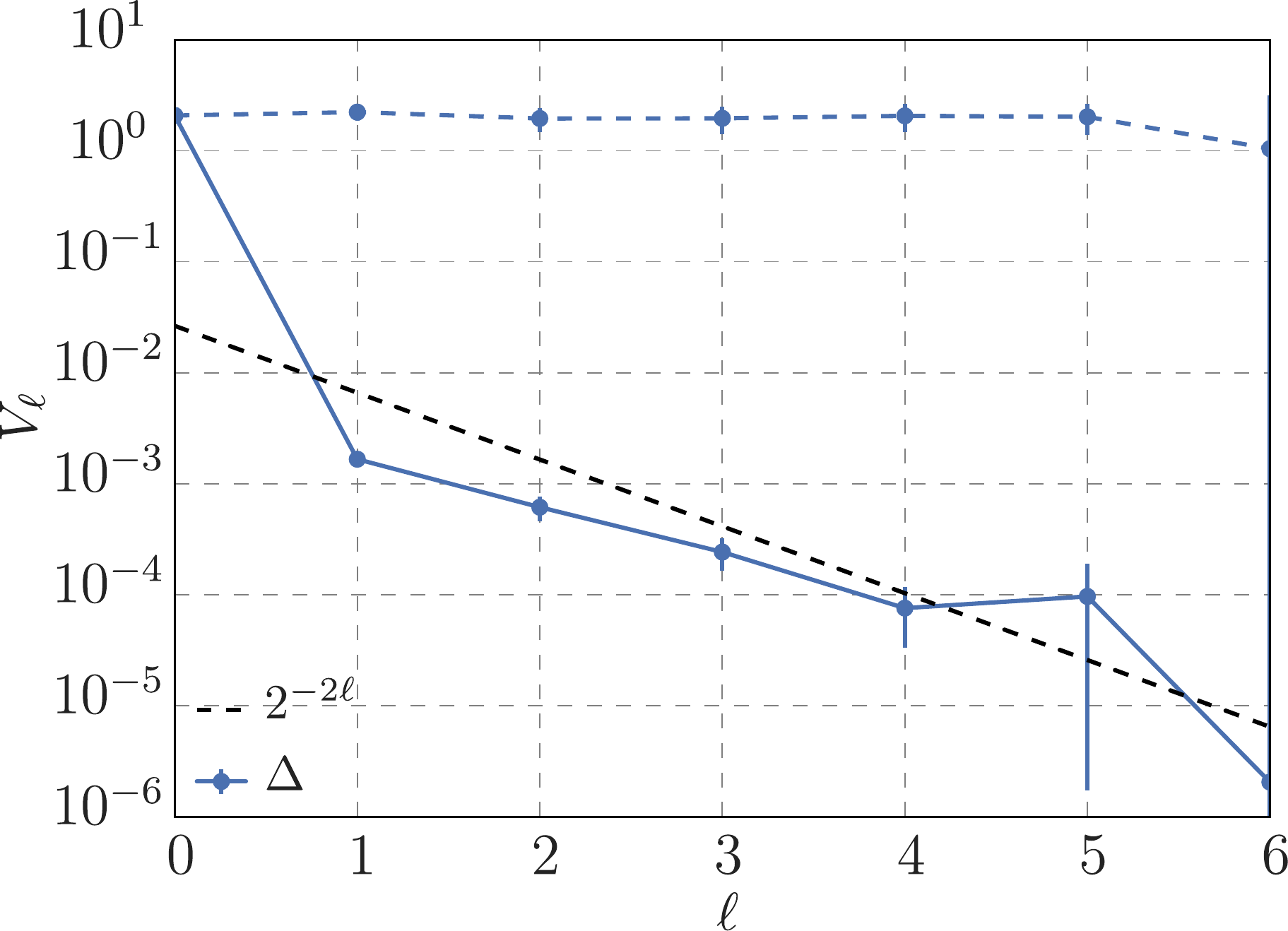}}
  \caption{$E_\ell$ and $V_\ell$ with respect to level $\ell$ for MLDLMC.}
  \label{fg:El_Vl_MLMC}
\end{figure}

\begin{figure}[H]
\centering
\includegraphics[width=0.5\textwidth]{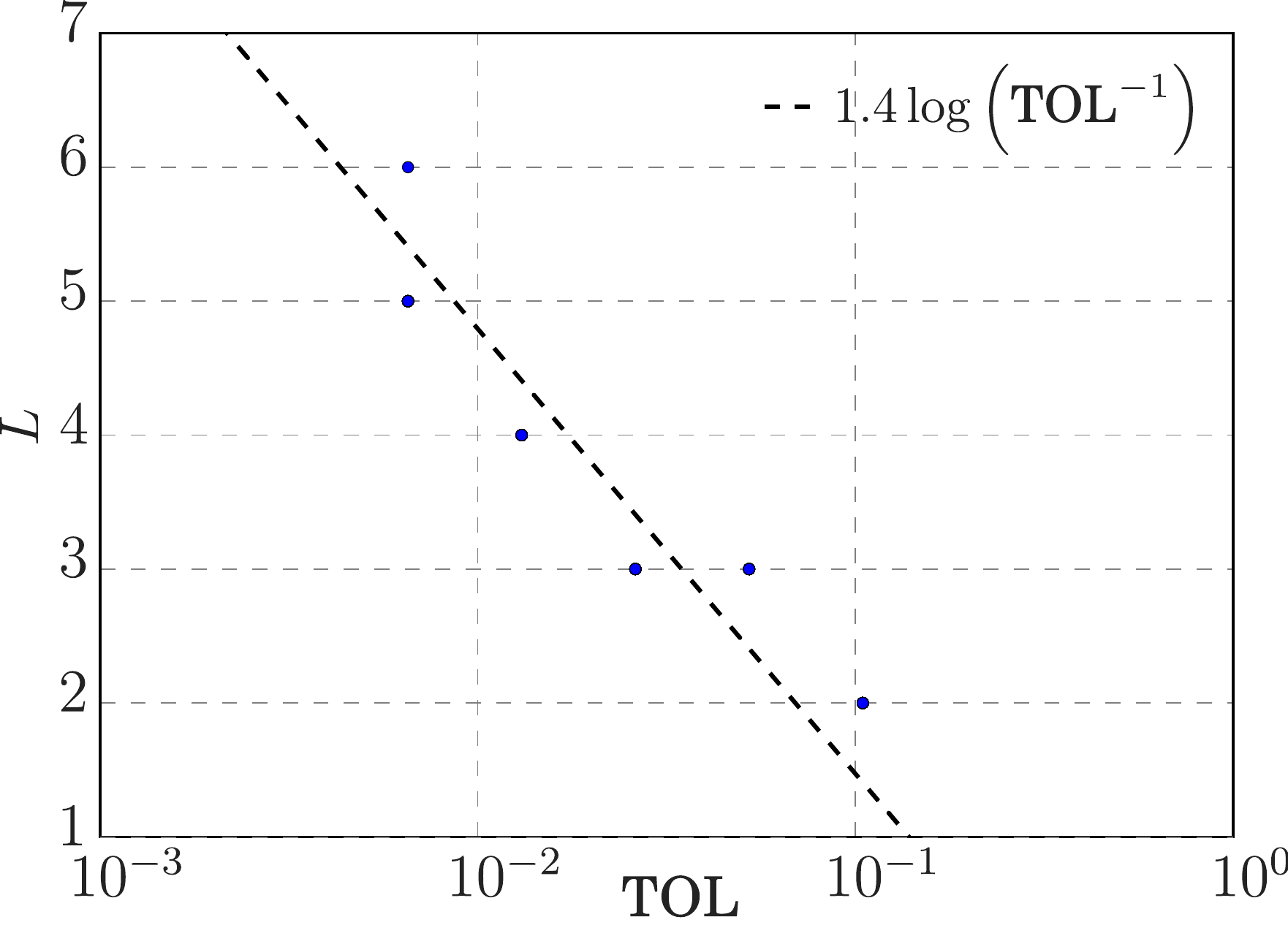}
\caption{Level $L$ of the MLDLMC with respect to $\hbox{TOL}$.}
\label{fg:LvsTOL}
\end{figure}

The finest mesh level $L$ considered in the MLDLMC method for the different choices of $\hbox{TOL}$ is shown in Figure \ref{fg:LvsTOL}. The finest level $L=6$ uses a mesh with $N_x=640$ and $N_y=256$. We observe that level $L$ follows approximately $L \approx 1.4\log\left( \hbox{TOL}^{-1} \right)$ in agreement with the asymptotic behavior of $L$ with respect to \hbox{TOL}, see \eqref{eq:Lstar}. In Figure \ref{fg:timevstol}, the computational times of MLDLSC and MLDLMC, with 100 repeated runs, are compared for a range of $\hbox{TOL}$. An estimate of the average computational time for the DLMCIS estimator \cite{BDELT2018} is also included to demonstrate the computational efficiency of the two proposed multilevel methods.

\begin{figure}[H]
\centering
\includegraphics[width=0.68\textwidth]{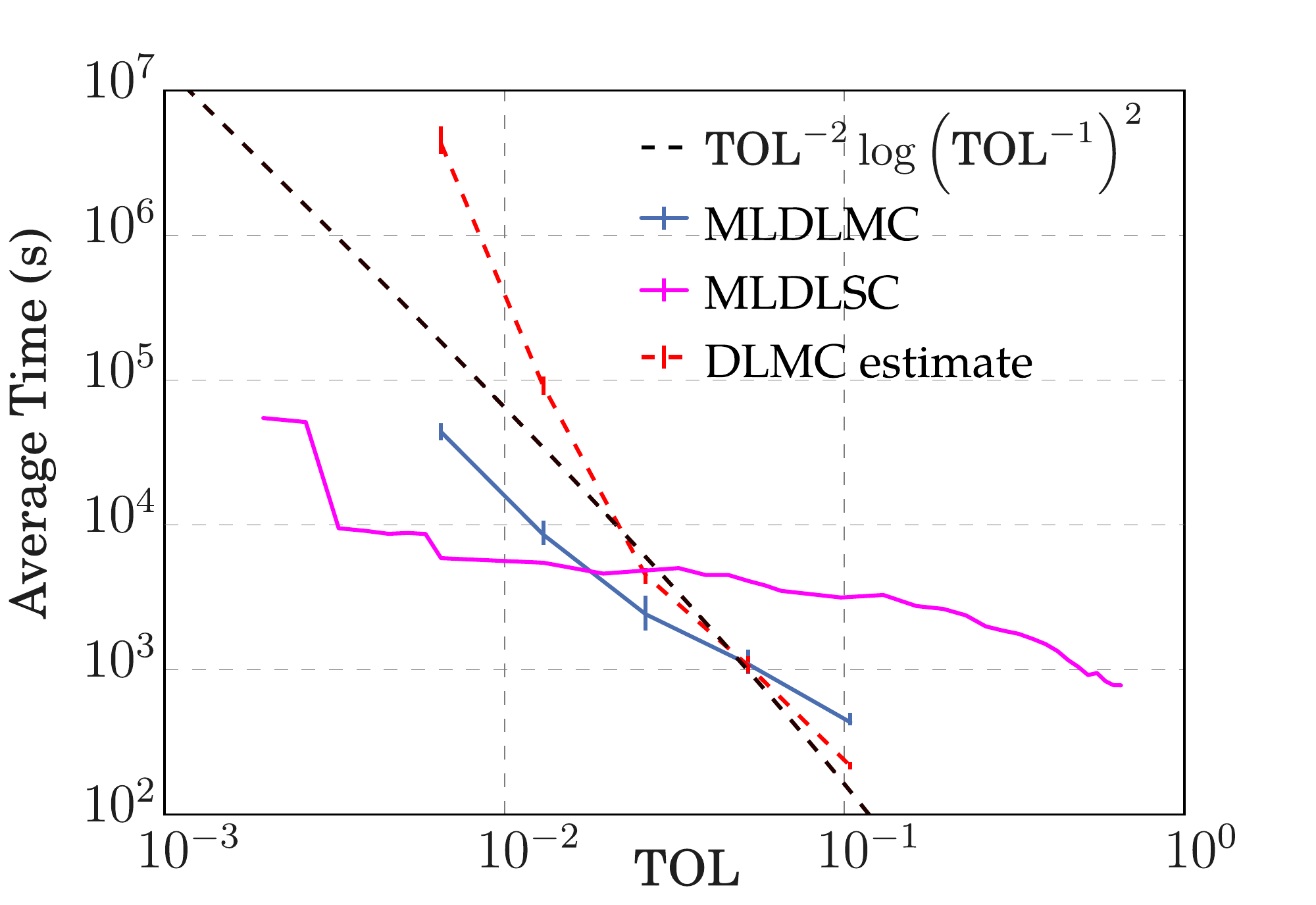}
\caption{Computational time for MLDLMC and MLDLSC, and the estimated time for DLMCIS.}
\label{fg:timevstol}
\end{figure}

MLDLSC exploits the regularity of the underlying model with respect to the uncertain parameter $\bm{\theta}$ and thus achieves a better work complexity than MLDLMC. Conversely, it should be mentioned that MLDLMC is expected to perform better when regularity is low. MLDLSC seems to perform worse than expected for large $\hbox{TOL}$, which could be attributed to the error estimates not being sharp. As shown in Section \ref{sec:compdiscuss}, the work complexity of DLMCIS follows that of the standard MLMC, in this case $\hbox{TOL}^{-2}\log\left( \hbox{TOL}^{-1} \right)^2$; see Theorem 2.1 \cite{G2015} with $\beta=\gamma$ where in that theorem $\beta \myeq 2\eta_s$. MLDLSC and MLDLMC are superior to DLMCIS, which shows that the multilevel construction greatly accelerates the computational performance.

\section{Conclusion}
In the situation when the experiments are modeled by PDEs, and under certain mild conditions on the underlying computational model, we propose two multilevel methods for computationally efficient estimation of the EIG criterion in the context of Bayesian optimal experimental design. The first method is the MLDLMC estimator, which is a multilevel extension of the DLMCIS method, and the second is the MLDLSC method, which uses an adaptive sparse-grid stochastic collocation scheme, and given enough regularity with respect to the random parameter could in many situations achieve a higher accuracy than MLDLMC at a lower computational cost. We found that the computational efficiency of using the multilevel construction relies strongly on keeping the number of inner samples to be low, which in the proposed methods are achieved by the Laplace-based importance sampling. In Bayesian optimal experimental design, the relative difference in EIG values between designs is of interest, and within the likely range of accuracies needed for such an estimation, the work complexity of MLDLMC is the same as for the standard MLMC.

We have demonstrated that the proposed multilevel estimators for the EIG criterion are computationally efficient, as they can balance the work over a hierarchy of meshes, and robust with reliable error control, i.e., to achieve a specified error tolerance with a high probability of success.

\section*{Acknowledgements}
The research reported in this publication was supported by funding from King Abdullah University of Science and Technology (KAUST) Office of Sponsored Research (OSR) under award numbers URF/1/2281-01-01 and URF/1/2584-01-01 in the KAUST Competitive Research Grants Program-Round 3 and 4, respectively, and the Alexander von Humboldt Foundation. J. Beck, L.F.R. Espath, and R. Tempone are members of the KAUST SRI Center for Uncertainty Quantification in Computational Science and Engineering.

\bibliographystyle{acm}
\bibliography{reference}

\end{document}